\documentclass[11pt,twoside,a4paper]{article}%

\input{packages.sty}
\usepackage{lmodern}

\usetikzlibrary{decorations.pathreplacing}
\usetikzlibrary{decorations.shapes}
\usetikzlibrary{calc}

\tikzset{decorate sep/.style 2 args=
{decorate,decoration={shape backgrounds,shape=circle,shape size=#1,shape sep=#2}}}

\title{Convergence Analysis of Grad's Hermite Expansion for Linear Kinetic Equations}
\author{Neeraj Sarna$^{1,*}$, Jan Giesselmann$^2$, Manuel Torrilhon$^1$}
\date{%
$^*${\small \texttt{sarna@mathcces.rwth-aachen.de}}\\
$^{1}$Center for Computational Engineering \&\ Department of Mathematics,\\RWTH\ Aachen University, Germany\\
$^{2}$Department of Mathematics,\\TU Darmstadt,\ Darmstadt, Germany\\
}

\begin{document}
\maketitle
\begin{abstract}
In (Commun Pure Appl Math 2(4):331-407, 1949), Grad proposed a Hermite series expansion for approximating solutions to kinetic equations that have an unbounded velocity space. However, for initial boundary value problems, poorly imposed boundary conditions lead to instabilities in Grad's Hermite expansion, which could result in non-converging solutions. For linear kinetic equations, a method for posing stable boundary conditions was recently proposed for (formally) arbitrary order Hermite approximations. In the present work, we study $L^2$-convergence of these stable Hermite approximations, and prove explicit convergence rates under suitable regularity assumptions on the exact solution. We confirm the presented convergence rates through numerical experiments involving the linearised-BGK equation of rarefied gas dynamics. 
\end{abstract}
\section*{Introduction}
Evolution of charged or neutral particles (under certain conditions of interaction) can be modelled by linear kinetic equations. The explicit form of these kinetic equations depends on the physical system they model and many of these forms have been extensively studied in the past; see \cite{Martin,Ukai,Grad1949,ExistenceLinearisedBE}. Broadly speaking, different forms of kinetic equations have mainly three differentiating factors: the space of possible velocities of particles, i.e., the so-called velocity space; the external or the internal forces that act on the particles; and the collision operator that models the interaction between different particles. 
In the present work, we are concerned with linear kinetic equations that have the whole $\mbb R^d$ ($1\leq d\leq 3$) as their velocity space, have no external force acting on the particles and have a collision operator that is bounded and negative semi-definite on $L^2(\mbb R^d)$. Such kinetic equations usually arise from the kinetic gas theory after the linearisation of the non-linear Boltzmann or the BGK equation \cite{Carlos}. 

Mostly, an exact solution to a kinetic equation is not known and one seeks an approximation through a temporal, spatial and velocity space discretization. In the present work, we analyse a Galerkin-type velocity space approximation where we approximate the solution's velocity dependence in a finite-dimensional space \cite{GambaPetrov,MomentsToDVM}. Our finite-dimensional space is the span of a finite number of Grad's tensorial Hermite polynomials, which results in the so-called Grad's moment approximation \cite{Grad1949}. We consider initial boundary value problems (IBVPs), and equip the Hermite approximation with boundary conditions that lead to its $L^2$-stability \cite{Sarna2018}.

The convergence behaviour of moment approximations, particularly for IBVPs, is not very well-understood. Lack of understanding originates from expecting a monotonic (and test case-independent) decrease in the error as the number of moments are increased but such a decrease is usually not observed in practise \cite{Torrilhon2015}.
It is known that convergence of Galerkin methods is solution's regularity dependent, which is in-turn test case dependent. Therefore,
one possible way to understand the test-case dependent convergence of moment approximations is to reformulate them as Galerkin methods \cite{ConvergenceMoments,Egger2016,Egger2012}. We use such a reformulation for the Grad's moment approximation to prove that it convergences (in the $L^2$-sense) to the kinetic equation's solution.  
%%%%%%%%%%%

Reformulation of a moment approximation as a Galerkin method allows us to use the following (standard) steps for convergence analysis. Firstly, we define a projection onto the Hermite approximation space and use it to split the approximation error into two parts: (i) one part containing the error in the expansion coefficients (or the moments), and (ii) the other part containing the projection error.
Secondly, we bound the error in the expansion coefficients in terms of the projection error. To develop this bound, we exploit the $L^2$-stability property of the Hermite approximation, which is possible by defining the projection such that it satisfies the same boundary conditions as those satisfied by the moment approximation. We complete our analysis by proving that the projection error converges to zero.

It is worth noting that the orthogonal projection onto the approximation space does not satisfy the same boundary condition as the numerical solution and, thus, the $L^2$-stability results are not available. Indeed, from a technical perspective, defining a suitable projection operator is a key contribution of this work.

In previous works \cite{MomentsToDVM,ConvergenceMoments}, for kinetic equations with an unbounded velocity space, authors have analysed convergence of Galerkin methods that use a grid in the velocity space. Although easier to implement, such methods fail to preserve the Galilean and the rotational invariance of kinetic equations. In contrast, Grad's tensorial Hermite polynomials cannot be mapped to a velocity space grid but they do preserve especially rotational invariance of kinetic equations. This allows for an approximation that is physically more sound. To the best of our knowledge, present work is the first step towards analysing the convergence of a rotational invariant Galerkin method for IBVPs involving kinetic equations with an unbounded velocity domain. 

Other approximation schemes that lead to a rotational invariant approximation (for both bounded and unbounded velocity spaces) use spherical harmonics instead of Grad's Hermite polynomials; see \cite{PNIntro,Christian,Egger2016}. Preliminary analysis shows that our framework is extendable to such approximations. Indeed, using our current framework one can even analyse the convergence of a general rotational invariant Galerkin scheme for a general rotational invariant kinetic equation considered in \cite{abstractKinetic}. Moreover, our framework has an extension to linear approximations of the non-linear Boltzmann equation \cite{GambaPetrov}. We leave an extension of our framework to other linear kinetic equations as a part of our future work.

A summary of the article's structure is as follows: the first section discusses the kinetic equation and its Grad's moment approximation; the second section discusses the projection operator and contains the main convergence result; the fourth section discusses an example of the linear kinetic equation that arises from the kinetic gas theory and; the fifth section contains our numerical experiment.

\section{Linear Kinetic Equation}
With $f:(0,T)\times \Omega\times\mbb R^d\to\mbb R$ we represent the solution to our kinetic equation where $\Omega$ is the physical space, $(0,T)$ is a bounded temporal domain and $\mbb R^d$ is the velocity space. For simplicity, we focus most of our discussion on the case for which the spatial domain is the open half-space $\Omega :=\mbb R^-\times \mbb R^{d-1}$ ($1\leq d\leq 3$). In \autoref{extension C2} we discuss how our framework can be extended to general $C^2$ spatial domains. With $V := (0,T)\times \Omega$ we represent the space-time domain and with $D := V\times \mbb R^d$ we represent our space-time-velocity domain. With $\nabla_{t,x}:=(\pd_t,\pd_{x_1},\dots,\pd_{x_d})$ we denote the gradient operator along the space-time domain and using it we define the following operator 
\begin{equation}
\begin{aligned}
\DiffOp:= &\pd_t + \sum_{i=1}^d\xi_i\pd_{x_i}- Q,\hsp\xi\in\mbb R^d,\\
=& (1,\xi)\cdot \nabla_{t,x} - Q, \label{def L}
\end{aligned}
\end{equation}
where $Q : L^2(\mbb R^d) \rightarrow L^2(\mbb R^d)$ is the collision operator.
The second form of the above operator will be helpful in understanding the regularity of a strong solution of an IBVPs involving $\DiffOp$. We restrict our analysis to the case for which the operator $Q$ satisfies the conditions enlisted below.
Later, in \autoref{an example}, we give examples of collision operators that satisfy the assumption below.
\begin{assumption}\label{assumption Q}
We assume that $Q : L^2(\mbb R^d) \rightarrow L^2(\mbb R^d)$ is: (i) linear, (ii) bounded, (iii) negative semi-definite, and (iv) self-adjoint. 
\end{assumption}
\noindent

 We consider $\mcal L$ as a mapping from $\OpSp$ to $L^2(D)$ where $\OpSp$ is the graph space of $\mcal L$ and is defined as
\begin{gather}
\OpSp := \{v\in \spaceXVT\hsp :\hsp \DiffOp v \in \spaceXVT\}\hspB\text{where}\hspB\|f\|_{\OpSp}^2 := \|f\|^2_{\spaceXVT} + \|\mcal L f\|^2_{\spaceXVT}.\label{def OpSp}
\end{gather}
For IBVPs involving the operator $\DiffOp$, we need to define trace operators over $\OpSp$. To define these trace operators, we first define the following boundaries of the set $D= (0,T)\times \Omega \times \mbb R^d$
\begin{gather*}
\Sigma^{\pm} := (0,T)\times \bcXi^{\pm},\hspB V^{\pm}:=\{T^{\pm}\}\times \Omega\times\mbb R^d,\hspB  \pd D := \Sigma^{+}\cup\Sigma^-\cup V^{+}\cup V^-,
\end{gather*}
where we set $T^+ = T$ and $T^- = 0$.
Moreover, $\bcXi^{\pm}$ is a result of splitting $\bc\times\mbb R^d$ into two non-overlapping parts and is defined as: $\bcXi^{\pm}:=\bc\times\mbb R^{\pm}\times\mbb R^{d-1}$.
Thus $\bcXi^+$ and $\bcXi^-$ are sets containing points in $\bc\times\mbb R^d$ corresponding to outgoing and incoming velocities, respectively. Using these boundary sets, in the following we define the relevant trace operators. A detailed derivation of these operators can be found in \cite{Ukai}. 
\begin{definition} \label{def: traces}
Traces of functions in $\OpSp$ are well-defined in $L^2(\pd D,|\xi_1|)$, i.e., in the $L^2$ space of functions over $\pd D$ with the Lebesgue measure weighted with $|\xi_1|$. We denote the trace operator by $$\gamma_D:\OpSp\to L^2(\pd D,|\xi_1|).$$
To restrict $\gamma_D$ to $\Sigma^{\pm}$ and $\Sigma=\Sigma^{+}\cup\Sigma^-$, we define
$\gamma^{\pm} f = \gamma_Df \rvert_{\Sigma^{\pm}}$ and $\gamma f = \gamma_Df \rvert_{\Sigma}$. 
Similarly, we interpret $f(T^{\pm})$ as 
$
f(T^{\pm}) = \gamma_Df\rvert_{V^{\pm}}$.
\end{definition}
Using the above trace operators, we give the following IBVP 
\begin{align}
\mcal L f = 0\hspB \text{in}\hspB D,\hspB
f(0) = f_{I}\hspB \text{on} \hspB V^-,\hspB
\gamma^-f = \fIn \hspB&\text{on}\hspB \Sigma^-, \label{BE}
\end{align}
where $f_I\in\spaceXV$ and $\fIn\in L^2(\Sigma^-;|\xi_1|)\cap L^2(\mbb R^-\times\mbb R^{d-1};H^{1/2}(\bc\times (0,T)))$ are some suitable initial and boundary data, respectively. Here $H^{\frac{1}{2}}$ denotes a standard fractional Sobolev space. The reason behind assuming $f_I$ to be in $\spaceXV$ and $\fIn$ to be in $\spaceXVMinus$ is clear from the definition of trace operators whereas, the assumption that $\fIn \in L^2(\mbb R^-\times\mbb R^{d-1};H^{1/2}(\bc\times (0,T)))$ will be made clear in \autoref{regularity assumption}.

We stick to strong solutions of the above IBVP and we define them as follows \cite{Ukai}.

\begin{definition}\label{def strong sol}
Let $f\in\OpSp$ where $\OpSp$ is as given in \eqref{def OpSp}. Then, $f$ is a strong solution to the linear kinetic equation if it satisfies 
\begin{gather*}
\lan v,\mcal Lf\ran_{\spaceXVT} = 0,\hspB \forall\hspB v\in\spaceXVT,\hspB \gamma^-f = \fIn,\hspB f(0)=f_I. 
\end{gather*}
\end{definition}
\noindent
It has been shown in \cite{Ukai} that the IBVP \eqref{BE} has a unique strong solution and for our convergence analysis, we will make additional regularity assumptions on this strong solution. 
We start with defining the notion of moments. 

\subsection{Moments and Hermite polynomials}
We define tensorial Hermite polynomials with the help of the multi-index $\multiIndex{i}$ as
\begin{equation}
\begin{gathered}
\basisPsi{i}(\xi):=\displaystyle\prod_{p=1}^d He_{\beta_p^{(i)}}\left(\xi_p\right),\hspB \beta^{(i)} := \left(\beta_1^{(i)},\dots,\beta_d^{(i)}\right),
 \label{def psi}
\end{gathered}
\end{equation}
where, the Hermite polynomials ($He_k$) enjoy the property of orthogonality and recursion 
\begin{subequations}
\begin{gather}
\frac{1}{\sqrt{2\pi}}\int_{\mbb R} He_i\left(\xi\right)He_j\left(\xi\right)\exp\left(-\frac{\xi^2}{2}\right)d\xi = \delta_{ij}\hspB \Rightarrow\hspB \int_{\mbb R^d}\psi_{\beta^{(k)}}\psi_{\beta^{(l)}}f_0 d\xi = \displaystyle\prod_{p=1}^d\delta_{\beta_p^{(k)}\beta_p^{(l)}},\label{orthonormality}\\
 \sqrt{i + 1}He_{i+1}\left(\xi\right) + \sqrt{i}He_{i-1}\left(\xi\right) = \xi He_i\left(\xi\right).	\label{recursion}
\end{gather}
\end{subequations}
Above, $f_0$ is a Gaussian weight given as 
\begin{gather}
f_0(\xi) := \exp\left(-\xi\cdot \xi/2\right)/\sqrt[d]{2\pi}. \label{def f0}
\end{gather}
The quantity $\normBeta{i}$ is the so-called degree of the basis function $\basisPsi{i}$. Below we define the $\normBeta{i}$-th order moment of a function in $\spaceV$.
\begin{definition} \label{def moments}
Let $n(m)$ represent the total number of tensorial Hermite polynomials (i.e. $\basisPsi{i}(\xi)$) of degree $m$ and let $\basis{m}(\xi)\in\mbb R^{n(m)}$ represent a vector containing all of such basis functions. 
Using $\basis{m}(\xi)$, we define $\mom{m}:\spaceV\to\mbb R^{n(m)}$ as:
$
\mom{m}(r) = \intV{\sqrt{f_0}\basis{m}(\xi)r(\xi)}$, $ \forall r\in \spaceV.$
Thus, $\mom{m}(r)$ represents a vector containing all the $m$-th order moments of $r$. To collect all the moments of $r$ which are of order less than or equal to $M$ ($m\leq M$), we additionally define
\begin{align*}
\CollectBasis{M}(\xi) = \left(\basis{0}(\xi)',\basis{1}(\xi)',\dots ,\basis{M}(\xi)'\right)'
,\quad
\CollectMom{M}(r) = \left(\mom{0}(r)',\mom{1}(r)',\dots ,\mom{M}(r)'\right)',
\end{align*}
where $\CollectBasis{M}(\xi)\in\mbb R^{\Xi^M}$ and $\CollectMom{M}:\spaceV\to \mbb R^{\Xi^M}$ with $\Xi^M = \sum_{m=0}^Mn(m)$ being the total number of moments. Above and in all of our following discussion, prime ( $'$ ) over a vector will represent its transpose.
\end{definition}
\subsection{Regularity Assumptions} \label{sec: regularity assumption}
For further discussion we recall that $V=\Omega\times(0,T)$ and $D=V\times\mbb R^d$. With $C^k([0,T];X)$ we denote a $k$-times continuously differential function of time with values in some Hilbert space $X$. We equip $C^k([0,T];X)$ with the norm $\|g\|_{C^k([0,T];X)} =\opn{max}_{j\leq k}\|\pd_t^j g\|_{C^0([0,T];X)}$ where $\|g\|_{C^0([0,T];X)} = \opn{max}_{t\in[0,T]}\|g(t)\|_X$. 

To capture velocity space regularity of solutions, we make use of the Hermite-Sobolev space $W_H^k(\mbb R^d)$ which is the image of $L^2(\mbb R^d)$ under the inverse of the Hermite Laplacian operator $\left(\Delta_H\right)^k = (-2\Delta + \frac{1}{2} \xi\cdot \xi)^{k}$; see \cite{HermiteSobolev} for details. One can show that a tensorial Hermite polynomial ($\basisPsi{m}$) is an eigenfunction of $\Delta_H$ with an eigenvalue of $(2m+d)$ and therefore, one can define norm of functions in $L^2(\Omega;W_H^k(\mbb R^d))$ as 
\begin{gather*}
\|f\|_{L^2(\Omega;W_H^k(\mbb R^d))} := \left(\sum_{m = 0}^{\infty}(2m+d)^{2k} \|\mom{m}(f(t,.,.))\|^2_{L^2(\Omega;\mbb R^{n(m)}))}\right)^{1/2}.
\end{gather*}
For further discussion we assume that the solution to our IBVP, along with its derivatives, lies in $C^0([0,T];L^2(\Omega;W_H^k(\mbb R^d)))$ for some $k$. We summarise this assumption in the following.

\begin{assumption} \label{regularity assumption}

Let $f$ be a strong solution to the kinetic equation \eqref{BE}. We assume that there exist numbers  $k^{e/o} \geq 0$, $k_{t}^{e/o}\geq 0$ and $k_{x}^{e/o}\geq 0$ such that
\begin{gather*}
f^{e/o}\in \SbHermite{k^{e/o}}{0},\hsp
\left(\pd_{t}f\right)^{e/o}\in \SbHermite{k_{t}^{e/o}}{0}, \\
\left(\pd_{x_i}f\right)^{e/o}\in \SbHermite{k_{x}^{e/o}}{0},\hspB \forall \hsp i\in\{1,\dots,d\}.
\end{gather*}
Above, $(.)^e$ and $(.)^o$ denote the even and odd parts (of the various quantities) defined with respect to $\xi_1$ i.e. $$f^o(\xi_1,\xi_2,\xi_3) = \frac{1}{2}\left(f(\xi_1,\xi_2,\xi_3) - f(-\xi_1,\xi_2,\xi_3)\right),\hsp f^e(\xi_1,\xi_2,\xi_3) = \frac{1}{2}\left(f(\xi_1,\xi_2,\xi_3) + f(-\xi_1,\xi_2,\xi_3)\right).$$
\end{assumption}
\noindent

Note that for simplicity we have assumed the same degree of regularity for all spatial derivatives. Extending the forthcoming results to cases where different spatial derivatives have different degrees of regularity is straightforward.

To understand the relation between a standard Sobolev space and the Hermite-Sobolev space, we recall the following result \cite{HermiteSobolev} (see Theorem 2.1) 
$$W_H^k(\mbb R^d)\subseteq H^{2k}(\mbb R^d)\subseteq L^2(\mbb R^d),\hspB \forall\hsp k \geq 0,$$ 
where $H^k(\mbb R^d)$ represents a standard Sobolev space and the last inclusion results from its definition.
Above relation and the assumption in \autoref{regularity assumption} trivially implies that the space-time gradient of $f$ (i.e. $\nabla_{t,x}f$) is in $L^2(D;\mbb R^{d+1})$ which further leads to 
\begin{gather}
f\in L^2(\mbb R^d; H^1(\Omega)) \cap \OpSp. \label{H1 regularity}
\end{gather}

Later, during the convergence analysis error terms will appear along the boundary ($\bc\times (0,T)$) involving the moments of the traces of $f$, i.e. $\mom{m}(\gamma f)$, and due to \autoref{regularity assumption} these error terms are well-defined. Indeed, $\mom{m}(\gamma f)$ is an element of $H^{\frac{1}{2}}(\bc\times (0,T);\mbb R^{n(m)})$. Note that for strong solutions, the moments of the traces are not necessarily well-defined. The fact that $\gamma f \in  L^2(\mbb R^d;H^{\frac{1}{2}}(\bc\times (0,T)))$ is required by our analysis is the reason why we assume the boundary data ($\fIn$ in \eqref{BE}) to be in $L^2(\Sigma^-;|\xi_1|)\cap L^2(\mbb R^-\times\mbb R^{d-1};H^{1/2}(\bc\times (0,T)))$, since for compatibility we want $\gamma^- f = f_{in}$ on $\Sigma^-$.

\subsection{Moment Approximation}
\subsubsection{Even and Odd basis functions: }
To formulate boundary conditions for our moment approximation (discussed next), we first need the notion of even and odd moments.
\begin{definition} \label{def muO}
Let $\numOdd(m)$ and $\numEven(m)$ denote the total number of tensorial Hermite polynomials in $\basis{m}(\xi)$ which are odd and even, with respect to $\xi_1$, respectively.
Similarly, let $\basisO{m}(\xi)\in\mbb R^{\numOdd(m)}$ and $\basisE{m}(\xi)\in\mbb R^{\numEven(m)}$ represent vectors containing those basis functions out of $\basis{m}(\xi)$ which are odd and even, with respect to $\xi_1$, respectively. Then, we define $\momO{m}:\spaceV \to\mbb R^{\numOdd(m)}$ and $\momE{m}:\spaceV\to\mbb R^{\numEven(m)}$ as:
$
\momO{m}(r) = \lan \basisO{m}\sqrt{f_0},r\ran_{\spaceV}$ and  $\momE{m}(r) = \lan \basisE{m}\sqrt{f_0},r\ran_{\spaceV}$ where $r\in \spaceV$.
To collect all the odd and even moments of $r$ which have a degree less than or equal to $M$ ($m\leq M$), we define 
\begin{gather*}
\CollectBasisO{M}(\xi) = \left(\basisO{1}(\xi)',\basisO{2}(\xi)',\dots \basisO{M}(\xi)'\right)',\quad \CollectBasisE{M}(\xi) = \left(\basisE{0}(\xi)',\basisE{1}(\xi)',\dots \basisE{M}(\xi)'\right)',\\
\CollectMomO{M}(r) = \left(\momO{1}(r)',\momO{2}(r)',\dots \momO{M}(r)'\right)',\quad \CollectMomE{M}(r) = \left(\momE{0}(r)',\momE{1}(r)',\dots \momE{M}(r)'\right)',
\end{gather*}
where
$\CollectMomO{M}:\spaceV\to\mbb R^{\numOddTot{M}}$, $\CollectMomE{M}:\spaceV\to\mbb R^{\numEvenTot{M}}$, $\CollectBasisO{M}(\xi)\in \mbb R^{\numOddTot{M}}$ and  $\CollectBasisE{M}(\xi)\in \mbb R^{\numEvenTot{M}}$. We represent the total number of odd and even moments of degree less than or equal to $M$ through 
$
\Xi_o^M = \sum_{i=1}^M n_o(i)$ and $\Xi_e^M = \sum_{i=0}^M n_e(i)
$ respectively.
\end{definition}
\noindent
Expressions for boundary conditions become compact if we define the following matrices.
\begin{definition} \label{def Aoe}
We define 
\begin{gather*}
 \MatFluxSm{p}{r} = \lan\CollectBasisO{p}\xi_1\sqrt{f_0},\left(\basisE{r}\right)^' \sqrt{f_0}\ran_{\spaceV},\hspB
\MatFluxBi{p}{q} = \left(\MatFluxSm{p}{1},\MatFluxSm{p}{2},\dots ,\MatFluxSm{p}{q}\right).
\end{gather*}
We interpret $\lan\CollectBasisO{p}\xi_1\sqrt{f_0},\left(\basisE{r}\right)^' \sqrt{f_0}\ran_{\spaceV}$ as a matrix whose elements contain $L^2(\mbb R^d)$ inner product between different elements of vectors $\CollectBasisO{p}\sqrt{f_0}$ and $\xi_1\basisE{r}\sqrt{f_0}$. Therefore, $\MatFluxSm{p}{r}$ is a matrix with real entries of dimension $\numOddTot{p}\times\numEven(r)$. Moreover by definition, $\MatFluxSm{p}{r}$ are the different groups of columns of $ \MatFluxBi{p}{q}$ for $r\in \{1,\dots,q\}$. 
\end{definition}
\noindent
Recall that both $\CollectBasisE{q}(\xi)$ and $\basisE{q}(\xi)$ are vectors but $\CollectBasisE{q}(\xi)$ contains all those basis functions that have a degree less than or equal to $q$ whereas, $\basisE{q}(\xi)$ contains basis function of degree equal to $q$. Similar to the above matrices, we define the following matrices, which also contain the inner products between Hermite polynomials but on a half velocity space.
\begin{definition} \label{def B}
We define 
\begin{gather*}
\MatHalfSm{p}{r}= 2\lan {\CollectBasisO{p}\sqrt{f_0},\left(\basisE{r}\right)^'\sqrt{f_0}}\ran_{L^2(\mbb R^+\times \mbb R^{d-1})},\hspB
\MatHalfBi{p}{q} = \left(\MatHalfSm{p}{1},\MatHalfSm{p}{2},\dots, \MatHalfSm{p}{q}\right),
\end{gather*}
where $\MatHalfSm{p}{r}\in\mbb R^{\numOddTot{p}\times\numEven(r)}$. Similar to $\MatFluxSm{p}{r}$ defined above, $ \MatHalfSm{p}{r}\in\mbb R^{\numOddTot{p}\times\numEven(r)}$ are the different groups of columns of $ \MatHalfBi{p}{q}$ for $r\in \{1,\dots ,q\}$. 
\end{definition}

\subsubsection{Test and Trial Space: }
To approximate the strong solution (see \autoref{def strong sol}) to our kinetic equation \eqref{BE}, we use a Petrov-Galerkin type approach where we approximate the velocity dependence in the test space (i.e. $L^2(D)$) and in the solution space (i.e. $\OpSp$) through a finite Hermite series expansion \eqref{def psi}. 
Indeed, for our Petrov-Galerkin approach, we choose the following test ($X_M$) and the solution space ($H_M$)
\begin{equation}
\begin{gathered}
\left(L^2(\mbb R^d; H^1(V)) \cap \OpSp\right) \supset H_M:= \{\alpha\cdot \CollectBasis{M}\sqrt{f_0} \hsp : \hsp \alpha \in H^1(V;\mbb R^{\numTot{M}})\},\\ L^2(D)\supset X_M:= \{\alpha\cdot \CollectBasis{M}\sqrt{f_0} \hsp : \hsp \alpha \in L^2(V;\mbb R^{\numTot{M}})\}, \label{solsp testsp}
\end{gathered}
\end{equation}
where $\CollectBasis{M}$ is a vector containing all the Hermite polynomials up to a degree $M$, see  \autoref{def moments}. 
Since $\alpha\in H^1(V;\mbb R^{\numTot{M}})$, trivially, $H_M$ is a subset of $L^2(\mbb R^d; H^1(V))$, 
which means that our Galerkin method is conforming.
However, the fact that $H_M\subset \OpSp$ is not obvious and we prove it in the following result.
\begin{lemma}
Let $H_M$ be as defined in \eqref{solsp testsp} then, $H_M\subset \OpSp$.
\end{lemma}
\begin{proof}
Let $f\in H_M$. To prove our claim we need to show that $\DiffOp f\in L^2(D)$ for which we only need to show that $\xi\cdot \nabla_x f\in \spaceXVT$; definition of $H_M$ and boundedness of $Q$ on $L^2(\mbb R^d)$ already implies that $\pd_t f\in L^2(D)$ and $Q(f)\in L^2(D)$. We show that $\xi\cdot \nabla_x f\in \spaceXVT$ by proving that $\xi _i\pd_{x_i}f\in L^2(D)$ for all $i\in \{1,\dots ,d\}$. For brevity we consider $i=1$, for other values of $i$ result follows analogously. Computing $\|\xi_1 \pd_{x_1} f\|_{L^2(D)}^2$ by expressing $f$ as $f = \alpha \cdot \CollectBasis{M}\sqrt{f_0}$, we find 
\begin{equation*}
\begin{aligned}
\|\xi_1 \pd_{x_1} f\|_{L^2(D)}^2 = \|\left(\pd_{x_1}\alpha\right)^{'} A\pd_{x_1}\alpha\|_{L^2(V)} \leq C\|\pd_{x_1}\alpha\|^2_{L^2(V;\mbb R^{\numTot{M}})} <\infty,
\end{aligned}
\end{equation*}
where $A = \lan\CollectBasis{M}\sqrt{f_0},\xi_1^2 \CollectBasis{M}\sqrt{f_0}\ran_{\spaceV}$. Above, the first inequality is a result of each entry of $A$ being bounded and the last inequality is a result of $\alpha\in H^1(V;\mbb R^{\numTot{M}})$.
\end{proof}

\begin{remark}
Note that for the BGK and the Boltzmann collision operator (given in \autoref{an example}), $\sqrt{f_0}$ is the global equilibrium. Therefore, for both of these operators, an approximation in $H_M$ (given in \eqref{solsp testsp}) is equivalent to expanding around the global equilibrium. This ensures that there exists a finite $M$ such that
\begin{gather}
\opn{ker}(Q) \subseteq \opn{span}\{\basis{\beta^{(i)}}\sqrt{f_0}\}_{\|\beta^{(i)}\|_{l^1}=1,\dots,M}. \label{inclusion kernel}
\end{gather}
The equilibrium state of the kinetic equation belongs to $\opn{ker}(Q)$ and the above conditions allows one to compute the same numerically. Note that for the linearised Boltzmann and the BGK operator, the above condition holds for $M=2$ \cite{Carlos}.

Collision operators of practical relevance known to us have $\sqrt{f_0}$ (or $f_0$ depending on the scaling) as their global equilibrium. If the global equilibrium is different from $f_0$, say $\hat f_0$, then an expansion around $\hat f_0$ results in an approximation space different from $H_M$. If this approximation space has basis functions that satisfy the property of recursion \eqref{recursion}, orthogonality \eqref{orthonormality}, totality in $L^2(\mbb R^d)$, even/odd parity (given in \autoref{def muO}), etc., then we expect to have results similar to what we propose here. Considering a different approximation space is out of scope of the present work. 
\end{remark}

\subsubsection{Variational Formulation: }To develop our Galerkin approximation, in the definition of the strong solution (given in \autoref{def strong sol}), we restrict the test space and the trial space to $X_M$ and $H_M$, respectively. This provides
\begin{subequations}
\begin{align}
&\textit{Find $f_M\in H_M$ such that} \nonumber\\
&\lan v,\mcal Lf_M \ran_{\spaceXVT}  = 0,\hsp\forall\hsp v\in X_M,\hspB \CollectMom{M}(f_M(0)) = \CollectMom{M}(f_I)\hsp \text{on} \hsp \Omega,\label{petrov galerkin}\\
& 
\CollectMomO{M}(\gamma f_M)  = R^{(M)}\MatFluxBi{M}{M}  \CollectMomE{M}(\gamma f_M) + \bcInhomo(\fIn)\hsp\text{on}\hsp (0,T)\times\bc, \label{OBC}
\end{align}
\end{subequations}
where $R^{(M)} \in\mbb R^{\numOddTot{M}\times\mbb R^{\numOddTot{M}}}$ is a s.p.d matrix given as \cite{Sarna2017}
\begin{gather}
R^{(M)} = \MatHalfBi{M}{M-1}\left(\MatFluxBi{M}{M-1}\right)^{-1}. \label{Onsager Mat}
\end{gather}
Invertibility of the matrix $\MatFluxBi{M}{M-1}$ follows from the recursion relation \eqref{recursion} and is discussed in detail in appendix-\autoref{properties Aoe}. Moreover, $\bcInhomo:L^2(\mbb R^-\times\mbb R^{d-1})\to\mbb R^{\numOddTot{M}}$ is defined as: $\bcInhomo(\fIn) := \lan\CollectBasisO{M},\fIn\ran_{L^2(\mbb R^-\times\mbb R^{d-1})}$. Thus, $\bcInhomo(\fIn)$ is a vector containing all the half-space odd moments of $\fIn$.
The variational form in \eqref{petrov galerkin} and its initial condition follow trivially from the definition of a strong solution given in \autoref{def strong sol}. However, the derivation of boundary conditions \eqref{OBC} is more involved and one can find details of this derivation in \cite{Rana2016,Sarna2017,Sarna2018}. For brevity, we refrain from discussing these details here.

The Galerkin formulation \eqref{petrov galerkin} is $L^2$-stable and its stability results from the specific form of the boundary conditions given in \eqref{OBC}. Since stability will be crucial for developing error bounds, we present a brief derivation of the stability estimate. We choose $v$ as $f_M$ in \eqref{petrov galerkin}, consider (for simplicity) $\fIn = 0$, use the negative semi-definiteness of $Q$ and perform integration-by-parts on the space-time derivatives to find
 \begin{equation}
\begin{aligned}
\|f_M(T)\|^2_{\spaceXV}-\|f_M(0)\|^2_{\spaceXV}  \leq  &- 2\lan\CollectMomO{M}(\gamma f_M),\MatFluxBi{M}{M}\CollectMomE{M}(\gamma f_M)\ran_{L^2((0,T)\times\bc;\mbb R^{\numOddTot{M}})}\\
=& - 2\lan \MatFluxBi{M}{M}\CollectMomE{M}(\gamma f_M),R^{(M)}\MatFluxBi{M}{M}\CollectMomE{M}(\gamma f_M)\ran_{L^2((0,T)\times\bc;\mbb R^{\numOddTot{M}})}\\
\leq &0,\label{stability fM}
\end{aligned}
\end{equation}
where the last inequality is a result of $\OnsagerMat{M}$ being s.p.d and all the boundary integrals are well-defined because $\CollectMom{M}(\gamma f_M) \in L^2(V;\mbb R^{\numTot{M}})$, which is a result of our definition of $H_M$ given in \eqref{solsp testsp}. Moreover, the integral on the boundary involving $\MatFluxBi{M}{M}$ results from the following, which results from the orthogonality of even and odd Hermite polynomials
\begin{equation}
\begin{aligned}
\int_{\mbb R^d}\xi_1 (\gamma f_M)^ 2 d\xi = &2\int_{\mbb R^d}\xi_1 (\gamma f_M)^o(\gamma f_M)^ed\xi \\
= &2\int_{\mbb R^d} \left(\CollectMomO{M}(\gamma f_M)\cdot \CollectBasisO{M}(\xi)\sqrt{f_0}\right) \xi_1 \left(\CollectBasisE{M}(\xi)\cdot \CollectMomE{M}(\gamma f_M)\sqrt{f_0}\right)d\xi \\ 
 =&2\lan \CollectMomO{M}(\gamma f_M),\MatFluxBi{M}{M}\CollectMomO{M}(\gamma f_M)\ran_{\mbb R^{\numOddTot{M}}}. \label{even odd ortho}
\end{aligned}
\end{equation}
\begin{remark}
The variational form in \eqref{petrov galerkin} is the same that leads to the Grad's moment equations \cite{Grad1949}. However, through \eqref{petrov galerkin}, we only recover the so-called \textit{full moment} approximations \cite{CaiNRXX,Torrilhon2015}.
\end{remark}

\begin{remark}\label{explain diff MBC OBC}
Grad \cite{Grad1949} prescribes boundary conditions through 
$
\CollectMomO{M}(\gamma f_M)  = \MatHalfBi{M}{M}  \CollectMomE{M}(\gamma f_M) + \bcInhomo(\fIn)
$ but they lead to $L^2$-instabilities \cite{Sarna2018,Rana2016}. To see the difference between Grad's boundary conditions and those which lead to stability \eqref{OBC}, we use the expression for $\OnsagerMat{M}$ from \eqref{Onsager Mat} and subtract the boundary matrix in \eqref{OBC} with the Grad's boundary matrix to find
\begin{align}
R^{(M)}\MatFluxBi{M}{M} - \MatHalfBi{M}{M} = \left(0, \left[R^{(M)}\MatFluxSm{M}{M}-\MatHalfSm{M}{M}\right]\right). \label{diff MBC OBC}
\end{align}
The above relation implies that the two boundary conditions differ only in terms of the highest order even moments of $f_M$ i.e. through $\momE{M}(f_M(t,x,.))$. This difference will show up in the convergence analysis and will influence the convergence order of our moment approximation.
\end{remark}

\begin{remark}
In \cite{Egger2016}, authors consider an IBVP for the radiative transport equation and develop a $L^2$-stable moment approximation  for the same. 
Comparing our approach to that proposed in \cite{Egger2016} is ongoing research and we hope to cater to it in the future. The framework proposed in \cite{Egger2016} considers a bounded velocity domain, which does not have a radial direction. Therefore, the first step is to extend this framework to an unbounded velocity domain, and then to compare it to ours.
\end{remark}

\section{Convergence Analysis}
We outline the forthcoming convergence analysis in the following steps. 
\begin{enumerate}[label=(\roman*)]
\item  \textit{Define a Projection Operator:} we define a projection operator $\projBC{M}:L^2(\mbb R^d;H^1(V))\to H_M$ (with $H_M$ as defined in \eqref{solsp testsp}) such that the trace of the projection satisfies the same type of boundary conditions as those satisfied by the moment approximation \eqref{OBC}. Such a projection operator helps us exploit the stability of the moment approximation \eqref{stability fM} during error analysis.
\item \textit{Decompose the error:} we decompose the moment approximation error into two parts
\begin{align}
E_M = f - f_M = \underbrace{f-\projBC{M}f}_{\projError{M}} + \underbrace{\projBC{M}f-f_M}_{e_M} \label{split error}.
\end{align}
Above, $e_M$ is the error in moments (or the expansion coefficients) and $\projError{M}$ is the projection error.

\item \textit{Bound for the projection error:} 
we derive a bound for $\|\projError{M}\|_{\spaceXVT}$ in terms of the moments of the solution, and using our regularity assumption (see \autoref{regularity assumption}) we show that $\|\projError{M}\|_{\spaceXVT}\to 0$ as $M\to\infty$.

\item \textit{Bound for the error in moments:} 
Using stability of our moment approximation \eqref{stability fM}, we bound $\|e_M\|_{\spaceXVT}$ in terms of $\|\DiffOp\projError{M}\|_{\spaceXVT}$, where $\mcal L$ is the projection operator. We complete the analysis by showing that $\|\DiffOp\projError{M}\|_{\spaceXVT}\to 0$ as $M\to\infty$. 
\end{enumerate}

\subsection{The Projection Operator}
We sketch our formulation of the projection operator $\projBC{M}:L^2(\mbb R^d;H^1(V))\to H_M$. Let $r\in L^2(\mbb R^d;H^1(V))$. We represent the projection $\projBC{M}r$ generically through $\projBC{M}r = \left(\hatCollectMomO{M}(r)\cdot \CollectBasisO{M} +\hat{\Lambda}_M^e(r)\cdot\CollectBasisE{M}(r) \right)\sqrt{f_0}$ where $\hatCollectMomO{M}$ and $\hat{\Lambda}_M^e$ are linear functionals defined over $L^2(\mbb R^d)$. For now assume that $\projBC{M}r \in H_M$ and that the trace of the projection (i.e. $\gamma \projBC{M}r$) is such that $\gamma(\projBC{M}r)  = \left(\hatCollectMomO{M}(\gamma r)\cdot \CollectBasisO{M} +\hat{\Lambda}_M^e(\gamma r)\cdot\CollectBasisE{M} \right)\sqrt{f_0}$. Once we define $\hatCollectMomO{M}$ and $\hat{\Lambda}_M^e$, it will be trivial that both of these assumptions are satisfied. As mentioned earlier, we want $\gamma(\projBC{M}r)$ to satisfy moment approximation's boundary conditions \eqref{OBC}. Since these boundary conditions have no restriction over the even moments, we choose $\hat{\Lambda}_M^e(r)$ to be the same as the even moments of $r$ i.e. we choose $\hat{\Lambda}_M^e(r)=\CollectMomE{M}(r)$. However, coefficients of the odd basis functions are constrained by moment approximation's boundary conditions \eqref{OBC} and thus we choose them as $\hatCollectMomO{M}(r) = \OnsagerMat{M}\MatFluxBi{M}{M}\CollectMomE{M}(r) + \bcInhomo(r)$. Such a choice of $\hatCollectMomO{M}(r)$ ensures that, provided the inflow part of $r$ coincides with $\fIn$, we have $\hatCollectMomO{M}(\gamma r) = \OnsagerMat{M}\MatFluxBi{M}{M}\CollectMomE{M}(\gamma r) + \bcInhomo(\fIn)$  along the boundary, i.e. the projection satisfies the boundary conditions of the moment approximation \eqref{OBC}. In the following, we summarise our projection operator and, for convenience, we also define the orthogonal projection operator. 
\begin{definition} \label{def projector}
%Let $r\in L^2(\mbb R^d;H^1(V))$ then, w
We define $\projBC{M}:L^2(\mbb R^d;H^1(V))\to H_M$ as 
\begin{equation*}
\begin{aligned}
r(\cdot) \mapsto \left(\hatCollectMomO{M}(r)\cdot \CollectBasisO{M}(\cdot) +\CollectMomE{M}(r)\cdot \CollectBasisE{M}(\cdot)\right) \sqrt{f_0(\cdot)}\hspB
\text{with}\hspB 
\hatCollectMomO{M}(r) := &R^{(M)}\MatFluxBi{M}{M} \CollectMomE{M}(r) + \bcInhomo(r). 
\end{aligned}
\end{equation*}
Similarly, with $X_M$ as given in \eqref{solsp testsp}, we define the orthogonal projection operator $\projOrtho{M}:L^2(D)\to X_M$ as 
\begin{gather*}
(\projOrtho{M}r) (\xi) = \left(\CollectMomO{M}(r)\cdot \CollectBasisO{M}(\xi) + \CollectMomE{M}(r)\cdot \CollectBasisE{M}(\xi)\right)\sqrt{f_0(\xi)},\hspB r\in \spaceXVT.
\end{gather*}
\end{definition}

\begin{remark}\label{initial conditions}
In \eqref{petrov galerkin}, we prescribe the initial conditions using the orthogonal projection operator, but there is no unique way of doing so. Our convergence analysis covers all projection or interpolation operators which introduce errors that decay at least as fast as the moment approximation error ($E_M$). Upcoming convergence analysis will clarify the fact that both $\projBC{M}$ and $\projOrtho{M}$ satisfy these criteria.
Therefore, for simplification, we prescribe the initial conditions through $f_M(0) = \projBC{M}f_I$, which ensures that $e_M(0) = 0$.
Note that implementing $\projBC{M}$ is cumbersome and therefore for implementation, one might want to prescribe initial conditions using $\projOrtho{M}$ or some other (easier to implement) interpolation.
\end{remark}

\begin{remark}
Due to our definition of the projection operator $\projBC{M}$, the projection error $\projError{M}$ (defined in \eqref{split error}) is not orthogonal to the approximation space $H_M$. This is in contrast to the analysis in \cite{ConvergenceMoments, Martin} where the use of an orthogonal projection operator leads to a $\projError{M}$ that is orthogonal to the approximation space.
\end{remark}

\subsection{Extension to spatial domains with $C^2$ boundaries: } \label{extension C2}
Velocity perpendicular to our spatial domain's boundary is $\xi_1$ and we have defined the projection operator ($\projBC{M}$) with respect to this velocity, this is implicit in the definition of the operators $\bcInhomo$ and $\MatFluxBi{M}{M}$. Since for the half-space ($\Omega=\mbb R^-\times\mbb R^{d-1}$) the boundary normal is the same at every boundary point, the definition of the projection operator remains the same for all boundary points. However, for a spatial domain other than the half-space, the normal along the boundary varies which results in different boundary points having different projection operators. We briefly discuss a methodology to construct the projection operators for a $C^2$-domain, which can have a normal that varies along the boundary. 

Let $\Omega\subset \mbb R^d$ be a domain with a $C^2$ boundary. Then, for every point $x_0 \in \bc$ we can define a line which passes through $x_0$ and points towards the interior of the domain in the direction opposite to the normal at $x_0$ ($n(x_0)$): 
$
L_{x_0} := \{ x\in\Omega : x-x_0 = \alpha n(x_0),\alpha\in\mbb R^-\}.
$
Since the boundary is $C^2$, there exists some $\delta > 0$ such that 
$\Omega_{\delta} := \{x\in\Omega : \opn{dist}(x,\partial\Omega) \geq \delta \}$
has the property that no two lines $L_{x_0}$ and $L_{x_1}$, for any $x_0,x_1\in\pd \Omega$, intersect within $\Omega^c_{\delta}$.

Inside $\Omega_{\delta}$ we use the orthogonal projection $\projOrtho{M}$ whereas outside of $\Omega_{\delta}$ we proceed as follows. For every $x\in\Omega_{\delta}^c$ (by definition of $\Omega_{\delta}$) there exists a unique $x_0$ such that $x\in L_{x_0}$. Let $\projBC{M}^{x_0}$ denote the projection operator accounting for the boundary conditions at $x_0$. Then at $x$ we define the projection operator to be the linear combination of the projection operator which satisfies the boundary conditions, $\projBC{M}^{x_0}$, and the orthogonal projection operator $\projOrtho{M}$
$$ \projBC{M}^{x} := \left(1 - \frac{|x-x_0|}{\delta}\right)\projBC{M}^{x_0} + \frac{|x-x_0|}{\delta}\projOrtho{M}.$$
In this way, $x\mapsto \projBC{M}^x(f_M(.,x,.))$ satisfies the desired boundary conditions and is $C^1$.

\begin{remark}
We emphasize that the projection operator defined in \autoref{def projector} is an analytical tool defined such that the projection satisfies the same boundary conditions as those satisfied by the moment approximation. It is nowhere needed for computing the moment approximation. This is also clear from the variational formulation given in \eqref{petrov galerkin}, where we set to zero the orthogonal projection of the evolution equation onto the approximation space. 

%Implementation of the boundary conditions in \eqref{petrov galerkin} requires (apart from the moments) the matrices $R^M$ and $A_{\Psi}^{(M,M)}$, and the vector $\mcal G(f_{in})$. For complex domains, $A_{\Psi}^{(M,M)}$ changes to $A_{\Psi}^{(M,M)} O^n$, where $O^n$ is an orthogonal rotation matrix that depends on the normal vector $n\in\mbb R^d$ of the boundary; see \cite{Torrilhon2017} for an extensive discussion. The explicit form of $O^n$ is known and needs to be evaluated at every boundary quadrature point. Moreover, the sparse matrices $R^M$ and $A_{\Psi}^{(M,M)}$ need to be computed once and stored. Expressing $f_{in}$ as a series expansion, one can show that the vector $\mcal G(f_{in})$ is of the form $B_M O^n \Lambda_M(f_{in})$, where $B_M$ is a sparse matrix that needs be computed only once, $O^n$ is as described above and $\Lambda_M(f_{in})$ are the moments of $f_{in}$. 
\end{remark}

\subsection{Main Result}
In the following, we summarise our main convergence result. 
\begin{theorem} \label{main result}
We can bound the error in the moment approximation, $E_M = f - f_M$, as 
\begin{equation}
\begin{aligned}
\|E_M(T)\|_{\spaceXV} \leq \|f(T) - \projBC{M} f(T)\|_{\spaceXV}  + T
\left(A_1(T) +  \|Q\| A_2(T) + A_3(T)\right) \label{error bound main}
\end{aligned}
\end{equation}
where
\begin{subequations}
\begin{align}
A_1(T) = &\left(\diffBC{M}\|\momE{M}(\pd_{t} f)\|_{C^0([0,T];L^2(\Omega;\mbb R^{\numEven(M)}))}\right. \nonumber\\ 
&\left. + \sqrt{2}\sum_{\beta\in \{e,o\}}\frac{1}{(2(M+1)+d)^{k_t^{\beta}}}\|\left(\pd_{t} f\right)^o\|_{C^0([0,T];L^2(\Omega;W_H^{k_t^{\beta}}(\mbb R^d)))} \right),\label{def A1}\\
A_2(T) = &\left(\diffBC{M}\|\momE{M}(f)\|_{C^0([0,T];L^2(\Omega;\mbb R^{\numEven(M)}))}\right. \nonumber\\
&\left. + \sqrt{2}\sum_{\beta\in \{e,o\}}\frac{1}{(2(M+1)+d)^{k^{\beta}}}\|f^{\beta}\|_{C^0([0,T];L^2(\Omega;W_H^{k^{\beta}}(\mbb R^d)))}\right), \label{def A2}\\ 
A_3(T) = &\sum_{i=1}^d\left(\diffBC{M}\|\MatFluxBi{M}{M}\|_2\|\momE{M}(\pd_{x_i} f)\|_{C^0([0,T];L^2(\Omega;\mbb R^{\numEven(M)}))}\right.\nonumber\\
&\left. + \sqrt{(M+1)}\|\mom{M+1}(\pd_{x_i}f)\|_{C^0([0,T];L^2(\Omega;\mbb R^{n(M+1)}))}\right) \nonumber\\
&+ \frac{\|\MatFluxBi{M}{M}\|_2}{(2(M+1)+d)^{k_x^e}}\sum_{i=1}^d\|\left(\pd_{x_i}f\right)^e\|_{C^0([0,T];L^2(\Omega;W_H^{k_x^e}(\mbb R^d)))},\label{def A3}\\
\diffBC{M} =  &\|R^{(M)}\MatFluxSm{M}{M}-\MatHalfSm{M}{M}\|_2. \label{def diffBC}
\end{align}
\end{subequations}
As $M\to \infty$, we have the convergence rate
\begin{equation}
\begin{gathered}
\|E_M(T)\|_{\spaceXV} \leq \frac{C}{M^{\omega}},\hspB \omega = \min\left\{k^{e/o}-\frac{1}{2}, k_t^{e/o}-\frac{1}{2},k_{x}^e-1,k_{x}^o-\frac{1}{2}\right\}.\label{convg rate}
\end{gathered}
\end{equation}
\end{theorem}
\noindent
The motivation behind decomposing the right hand side into the different $A_i$'s is that each of these terms vanishes in different physical settings. The term $A_1$ vanishes for steady state problems i.e. for $\pd_t f = 0$, the term $A_2$ vanishes in the absence of collisions, and the term $A_3$ vanishes under spatial homogeneity i.e. for $\pd_{x_i} f = 0$.

An alternative way to understand the right hand side of the error bound given in \autoref{main result} is to identify the following four different types of errors:
\begin{enumerate}[label=(\roman*)]
\item \textit{Projection Error: }This is the first term appearing on the right side of the error bound in \eqref{error bound main} and is the $\projError{M}$ defined in \eqref{split error}.

\item \textit{Closure Error: }This is the second term appearing in $A_3(T)$ \eqref{def A3} and involves the $M+1$-th order moment of $\pd_{x_i}f$. The term accounts for the influence of the flux of the $M+1$-th order moment which was dropped out during the moment approximation.

\item \textit{Boundary Stabilisation Error: }These are all the terms involving $\diffBC{M}$ and are all the first terms appearing in \eqref{def A1}-\eqref{def A3}. These terms are a result of the difference between the boundary conditions proposed by Grad \cite{Grad1949} and those given in \eqref{OBC} which lead to a stable moment approximation; \autoref{explain diff MBC OBC} explains the difference between the two boundary conditions.
Since the two boundary conditions only differ in the coefficients of the highest order even moment (see \eqref{diff MBC OBC}), this error depends only upon this highest order even moment.

\item \textit{Boundary Truncation Error: }These are all the terms which are not included in the above definitions. They are a result of ignoring contributions from all those even (and odd) moments which have an order greater than $M$ and do not appear in the boundary conditions for the moment approximation \eqref{OBC}.
\end{enumerate}
We prove \autoref{main result} in the next few pages.
%%%% FIX FROM HERE %%%%%%%%%%.
\subsection{Error Equation}
To derive a bound for the moment approximation error \\
(i.e. for $\|E_M(T)\|_{\spaceXV}$) we first derive a bound for the error in the expansion coefficients (i.e. for $\|e_M(T)\|_{\spaceXV}$) and then use triangle's inequality to arrive at a bound for $\|E_M(T)\|_{\spaceXV}$; see \eqref{split error} for definition of $E_M$ and $e_M$. In the following discussion we suppress dependencies on $x$ and $\xi$, for brevity. 

We start with adding and subtracting $\mcal L(\projBC{M}f)$ in the definition of a strong solution given in \autoref{def strong sol}. For all $\hsp v\in X_M,$ and for all  $t\in (0,T)$, considering the integral over $\Omega\times\mbb R^d$ provides
\begin{equation*}
\begin{aligned}
\lan v(t), \mcal L(\projBC{M}f(t)) \ran_{\spaceXV}  = &\lan v(t), \mcal L(\projBC{M}f(t)-f(t)) \ran_{\spaceXV}\hsp, \\
= &\lan v(t), \projOrtho{M}\mcal L(\projBC{M}f(t)-f(t)) \ran_{\spaceXV},
\end{aligned}
\end{equation*}
where $X_M\subset L^2(D)$ is as defined in \eqref{solsp testsp}. For the last equality we have used the trivial relation: $
\lan v(t),w(t)\ran_{\spaceXV} = \lan v(t),\projOrtho{M}w(t)\ran_{\spaceXV},\forall (v,w)\in X_M\times \spaceXVT.$ 
Subtracting the above relation from our moment approximation \eqref{petrov galerkin}, and using the linearity of $\mcal L$, we find
\begin{equation}
\begin{aligned}
\lan v(t),\mcal L(e_M(t)) \ran_{\spaceXV}  = \lan v(t),\projOrtho{M}\mcal L(f(t)-\projBC{M}f(t)) \ran_{\spaceXV}\hsp \forall\hsp v\in X_M,\hsp \forall\hsp t\in (0,T), \label{error equation}
\end{aligned}
\end{equation}
where $e_M$ is as given in \eqref{split error}. To derive a bound for $e_M$, we want to use the stability of our moment approximation \eqref{stability fM}. We do so by choosing $v(t) = e_M(t)$ in the above expression and by performing integration-by-parts on the spatial derivatives, which provides
\begin{multline}\label{aya}
\lan e_M(t),\pd_t e_M(t) \ran_{\spaceXV} - \lan e_M(t),Q e_M(t) \ran_{\spaceXV} \\ \leq  \lan e_M(t),\projOrtho{M}\mcal L(f(t)-\projBC{M}f(t)) \ran_{\spaceXV}-\underbrace{\oint_{\bc} \int_{\mbb R^d}\xi_1 (\gamma e_M(t))^2d\xi ds}_{\geq 0}.
\end{multline}
Later (in \autoref{an example}) we present physically relevant examples where the non-dimensionalisation of the kinetic equation results in the so-called Knudsen number, the inverse of which scales the collision operator. Depending on whether or not we are interested in the low Knudsen number regime, we can proceed with the above bound in different ways. Here we consider a Knudsen number that is large enough and postpone the discussion of small Knudsen numbers to \autoref{uni Knud}. Since $Q$ is negative semi-definite, using the Cauchy-Schwartz inequality to the above bound provides
\begin{equation}
\lan e_M(t),\pd_t e_M(t) \ran_{\spaceXV}   \leq \|e_M(t)\|_{L^2(\Omega\times\mbb R^d)} \|\projOrtho{M}\mcal L(f(t)-\projBC{M}f(t))\|_{L^2(\Omega\times\mbb R^d)}. \label{eqn eM}
\end{equation}
The integral over the boundary is positive because the trace of the projection (i.e $\gamma\projBC{M}f$) satisfies the same boundary conditions as those satisfied by our moment approximation \eqref{OBC}. To see this more clearly, consider the following relation which results from the even-odd decoupling \eqref{even odd ortho} and the moment equation's boundary conditions
\begin{equation*}
\begin{aligned}
\oint_{\bc}\int_{\mbb R^d}\xi_1 (\gamma e_M(t))^2d\xi ds = &\oint_{\bc}\left(\CollectMomO{M}(\gamma e_M(t))\right)^{'}\MatFluxBi{M}{M}\CollectMomE{M}(\gamma e_M(t))ds,\\ = & \oint_{\bc}\left(\CollectMomE{M}(\gamma e_M(t))\right)^{'}\left(\MatFluxBi{M}{M}\right)^{'} \OnsagerMat{M}\MatFluxBi{M}{M}\CollectMomE{M}(\gamma e_M(t))ds \geq 0.
\end{aligned}
\end{equation*}
The last inequality is a result of $\OnsagerMat{M}$ being \textit{s.p.d}.
Using the fact that $\lan e_M(t),\pd_t e_M(t) \ran_{\spaceXV} = \|e_M(t)\|_{L^2(\Omega\times\mbb R^d)}\pd_t \|e_M(t)\|_{L^2(\Omega\times\mbb R^d)}$ in \eqref{eqn eM}, dividing throughout by $\|e_M(t)\|_{L^2(\Omega\times\mbb R^d)}$ (result is trivial for $e_M = 0$) and integrating over time provides the following bound
\begin{equation}
\begin{aligned}
\|e_M(T)\|_{L^2(\Omega\times\mbb R^d)} \leq & \int_0^T \| \projOrtho{M}\mcal L(f(t)-\projBC{M}f(t)) \|_{\spaceXV}dt, \\
                                       \leq & T \|\projOrtho{M}\mcal L(f(t)-\projBC{M}f(t))\|_{C^0([0,T];\spaceXV)}.  
  \label{bound eM}
\end{aligned}
\end{equation}
Above, our choice of the initial conditions (see \autoref{initial conditions} ) results in  $e_M(0)=0$. To spell out the above term on the right, we use the definition of $\mcal L$ from \eqref{def L}, the boundedness assumption on $Q$ and triangle's inequality to find 
\begin{equation}
\begin{aligned}
\| \projOrtho{M}\mcal L(f(t)-\projBC{M}f(t)) \|_{\spaceXV} \leq &\|\pd_t f(t)-\projBC{M} \pd_t f(t)\|_{\spaceXV} + \|Q\|\|f(t)-\projBC{M}f(t)\|_{\spaceXV}\\& + \sum_{ i = 1}^d\|\projOrtho{M}\left(\xi_i\left(\pd_{x_i} f(t)-\projBC{M} \pd_{x_i}f(t) \right)\right)\|_{\spaceXV}.   \label{bound 0}
\end{aligned}
\end{equation}
We can further simplify $\|\projOrtho{M}\left(\xi_i\left(\pd_{x_i} f(t)-\projBC{M} \pd_{x_i}f(t) \right)\right)\|_{\spaceXV}$ by adding and subtracting\\
$\projOrtho{M}\xi_i \projOrtho{M} \pd_{x_i}f(t)$. Then, triangle's inequality provides 
\begin{equation}
\begin{aligned}
\|\projOrtho{M}\left(\xi_i\left(\pd_{x_i} f(t)-\projBC{M} \pd_{x_i}f(t) \right)\right)\|_{\spaceXV} \leq &\left(\|\projOrtho{M}\left(\xi_i\left( \projOrtho{M} \pd_{x_i}f(t)-\projBC{M}\pd_{x_i}f(t) \right)\right)\|_{\spaceXV}\right.\\
 &\left.+ \|\projOrtho{M}\left(\xi_i\left(\pd_{x_i} f(t)-\projOrtho{M} \pd_{x_i}f (t)\right)\right)\|_{\spaceXV}\right). \label{bound flux 1}
\end{aligned}
\end{equation}
To simplify the first term on the right we use (page-80, \cite{ConvergenceMoments})
\begin{align}
\|\projOrtho{M}\left(\xi_i\left( \projOrtho{M} \pd_{x_i}f(t)-\projBC{M}\pd_{x_i}f(t) \right)\right)\|_{\spaceXV}\leq \|\MatFluxBi{M}{M}\|_2\|\left( \projOrtho{M} \pd_{x_i}f(t)-\projBC{M}\pd_{x_i}f(t) \right)\|_{\spaceXV}. \label{bound flux 2}
\end{align}
Moreover, to simplify the second term on the right in \eqref{bound flux 1} we use the orthogonality and the recursion of Hermite polynomials to find 
\begin{equation}
\begin{aligned}
\|\projOrtho{M}\left(\xi_i\left(\pd_{x_i} f(t)-\projOrtho{M} \pd_{x_i}f(t) \right)\right)\|_{\spaceXV} = &\|\projOrtho{M}\left(\xi_i\left(\mom{M+1}(\pd_{x_i}f(t))\cdot\basis{M+1}\right)\sqrt{f_0}\right)\|_{\spaceXV}\\
 \leq &\sqrt{\left(M+1\right)}\|\mom{M+1}(\pd_{x_i}f(t))\|_{L^2(\Omega;\mbb R^{n(M+1)})}. \label{bound flux 3}
\end{aligned}
\end{equation}
Substituting \eqref{bound flux 1}-\eqref{bound flux 3} into \eqref{bound 0} and substituting the resulting expression into the bound for $e_M$, we find the following bound for $\|E_M(T)\|_{\spaceXV}$
\begin{equation}
\begin{aligned}
\|E_M(T)\|_{\spaceXV} \leq &\|f(T)-\projBC{M}f(T)\|_{\spaceXV} + \|e_M(T)\|_{\spaceXV} \\ 
										\leq &\|f(T)-\projBC{M}f(T)\|_{\spaceXV} 
										+ T \left(\tilde A_1(T) + \|Q\| \tilde A_2(T) + \tilde A_3(T)\right), \label{estimate EM}
\end{aligned}
\end{equation}
with
\begin{equation}
 \begin{aligned}
 \tilde A_1(T) &:= \|\pd_t f-\projBC{M} \pd_{t} f\|_{C^0([0,T];\spaceXV)},\\
   \tilde A_2(T) &:=\|f-\projBC{M}f\|_{C^0([0,T];\spaceXV)},\\
   \tilde A_3(T) &:=  \sqrt{(M+1)}\sum_{i=1}^d \|\mom{M+1}(\pd_{x_i}f)\|_{C^0([0,T];L^2(\Omega;\mbb R^{n(M+1)}))}\\
  &\quad+ \|\MatFluxBi{M}{M}\|_2\sum_{i=1}^d \|\projOrtho{M}\pd_{x_i} f-\projBC{M} \pd_{x_i}f \|_{C^0([0,T];\spaceXV)}. \label{def tilde A}
  \end{aligned}
 \end{equation}
The above expression is a bound for the moment approximation error in terms of the \textit{closure error} and the \textit{projection error} of different quantities.
Rate of convergence for the \textit{closure error} will trivially follow from the velocity space regularity assumption made in \autoref{regularity assumption}. Therefore, to complete our proof of \autoref{main result} we develop a bound for the norm of $\MatFluxBi{M}{M}$ and a bound for the \textit{projection error}. 
In particular, \autoref{convg projector} will show
\begin{equation}
 \tilde A_i(T) \leq A_i(T) \quad \text{for } i=1,2,3,
\end{equation}
where $A_i(T)$ are as defined in \autoref{main result}.
\subsection{Projection Error}
The following result shows that we can express the odd moments of any $r\in \spaceV$ in terms of its even moments and the function $\bcInhomo$ defined in \eqref{OBC}. The result will allow us to quantify the projection error in terms of the odd and the even moments of degree higher than $M$ which were left out while defining the projection operator $\projBC{M}$. 
\begin{lemma} \label{lemma:int split}
For every $r\in\spaceV$, it holds
\begin{align}
\lan\CollectBasisO{M}\sqrt{f_0}, r^o\ran_{L^2(\mbb R^d)} = 2\lan\CollectBasisO{M}\sqrt{f_0}, r^e\ran_{L^2(\mbb R^+\times\mbb R^{d-1})} +\bcInhomo(r), \label{int ro re}
\end{align}
or equivalently $
\CollectMomO{M}(r)  = \lim_{q\to\infty} \MatHalfBi{M}{q}\CollectMomE{q}(r) + \bcInhomo(r)$
where $r^o$ and $r^e$ are the odd and even parts of $r$, with respect to $\xi_1$, respectively, and $\bcInhomo$ is as given in \eqref{OBC}.
We interpret $\lim_{q\to\infty} \MatHalfBi{M}{q}\CollectMomE{q}(r)$ as $\lim_{q\to\infty} \left(\MatHalfBi{M}{q}\CollectMomE{q}(r)\right)$ where $\MatHalfBi{M}{q}$ is as given in \autoref{def B} and the limit is well-defined for all $r\in\spaceV$. 
\end{lemma}
\begin{proof}
See appendix-\autoref{proof lemma2.1}.
\end{proof}
\noindent
In the following result, we collect all the relevant bounds on different matrices and operators. We will use these bounds to formulate the convergence rate of the \textit{projection error}.
\begin{lemma} \label{normB}
\leavevmode
\begin{enumerate}[label=(\roman*)]
\item For $\lim_{q\to\infty}\MatHalfBi{M}{q}$ it holds
$\|\lim_{q\to\infty}\MatHalfBi{M}{q}\| \leq 1$
where $\lim_{q\to\infty}\MatHalfBi{M}{q}$ is as given in \autoref{lemma:int split}.
\item For $\MatFluxBi{M}{M}$ and $\MatFluxBi{M}{M-1}$ it holds:
$
\|\left(\MatFluxBi{M}{M-1}\right)^{-1}\MatFluxSm{M}{M}\|_2 \leq C \sqrt{M}$ and $ \|\MatFluxBi{M}{M}\|_2 \leq C \sqrt{M}.
$
\end{enumerate}
\end{lemma}
\begin{proof}
See appendix-\autoref{norm mat and op}. 
\end{proof}
Using the above results, in the following we develop a convergence rate and an error bound for the projection error. 
\begin{lemma} \label{convg projector}
Let $r^{e/o}\in C^0([0,T];L^2(\Omega;W_H^{k^{e/o}}(\mbb R^d)))$ then we can bound $\|\projBC{M}r(t)-r(t)\|^2_{\spaceXV}$ as 
\begin{align*}
\|\projBC{M} r(t)-r(t)\|^2_{\spaceXV} \leq &(\diffBC{M})^2 \|\momE{M}(r(t))\|^2_{L^2(\Omega;\mbb R^{\numEven(M)})} \\&+ 2\sum_{\beta\in \{e,o\}}\frac{1}{(2(M+1)+d)^{2k^{\beta}}}\|r^{\beta}(t)\|^2_{L^2(\Omega;W_H^{k^{\beta}}(\mbb R^d))},
\end{align*}
where $\diffBC{M} =  \|R^{(M)}\MatFluxSm{M}{M}-\MatHalfSm{M}{M}\|_2$ and dependency on $x$ and $\xi$ is hidden for brevity. Similarly, we can bound the difference between the orthogonal projection and the projection that satisfies the boundary conditions as
\begin{align*}
\|\projBC{M} r(t)-\projOrtho{M}r(t)\|^2_{\spaceXV} \leq (\diffBC{M})^2 \|\momE{M}(r(t))\|^2_{L^2(\Omega;\mbb R^{\numEven(M)})} + \frac{1}{(2(M+1)+d)^{2k^e}}\|r^e(t)\|^2_{L^2(\Omega;W_H^{k^e}(\mbb R^d))}. 
\end{align*}
As $M\to\infty$, we have the convergence rate
\begin{align*}
\|\projBC{M} r-r\|_{C^0([0,T];\spaceXV)} \leq CM^{-\td\omega},\hspB \|\projBC{M} r-\projOrtho{M}r\|_{C^0([0,T];\spaceXV)} \leq CM^{-(k^e-\frac{1}{2})},
\end{align*}
where $\td\omega=\min\left\{k^o-\frac{1}{2},k^e-\frac{1}{2}\right\}$.
\end{lemma}
\begin{proof}
We express $r$ in terms of tensorial Hermite polynomials and use \autoref{lemma:int split} to find 
\begin{align*}
r =\sum_{m=0}^M\left(\momO{m}(r)\cdot \basisO{m}(\xi) + \momE{m}(r)\cdot \basisE{m}(\xi)\right) \sqrt{f_0},\hsp \text{with} \hsp \CollectMomO{M}(r)  = \lim_{q\to\infty} \MatHalfBi{M}{q}\CollectMomE{q}(r) + \bcInhomo(r),
\end{align*}
where $\CollectMomO{M} = (\momO{1}(r)',\dots,\momO{M}(r)')$ and $\CollectMomE{M} = (\momE{0}(r)',\dots,\momE{M}(r)')$.
Moreover, the definition of $\projBC{M}r$ (see \autoref{def projector}) provides
\begin{gather*}
\projBC{M}r = \sum_{m=0}^M \left(\hatCollectMomO{m}(r)\cdot \CollectBasisO{m}(\xi) +\CollectMomE{m}(r)\cdot \CollectBasisE{m}(\xi)\right) \sqrt{f_0},\hsp\text{with}\hsp \hatCollectMomO{M}(r) =R^{(M)}\MatFluxBi{M}{M} \CollectMomE{M}(r) + \bcInhomo(r),
\end{gather*}
where $\hatCollectMomO{M} = (\hat{\lambda}_1^o(r)',\dots,\hat{\lambda}_M^o(r)')$.
Subtracting $r$ from $\projBC{M}r$, 
using 
$\lim_{q\to\infty} \MatHalfBi{M}{q}\CollectMomE{q}(r) = \displaystyle\sum_{q = 0}^{\infty} \MatHalfSm{M}{q} \momE{q}(r)$ and the simplified expression for $\OnsagerMat{M}\MatFluxBi{M}{M}-\MatHalfBi{M}{M}$ from \eqref{diff MBC OBC}, we find 
\begin{equation}
\begin{aligned}
\projBC{M}r-r = &\left((R^{(M)}\MatFluxSm{M}{M}-\MatHalfSm{M}{M})\momE{M}(r)\right)\cdot \basisO{M}(\xi)\sqrt{f_0} -  \sum_{q = M+1}^{\infty}\left(\MatHalfSm{M}{q}\momE{q}(r)\right)\cdot \basisO{M}(\xi)\sqrt{f_0} \\
& - \sum_{q = M+1}^{\infty}\left(\momE{q}(r)\cdot \basisE{q}(\xi) + \momO{q}(r)\cdot \basisO{q}(\xi)\right) \sqrt{f_0},\label{diff r and Pi r}
\end{aligned} 
\end{equation}
where $\MatHalfSm{M}{M}$ is as defined in \autoref{def B}. The matrices $\MatHalfSm{M}{q}$ and the operator $\lim_{q\to\infty}\MatHalfSm{M}{q}$ appearing above can be looked upon as restrictions of the operator $\lim_{q\to\infty}\MatHalfBi{M}{q}$ given in \autoref{normB}; thus all of their norms can be bounded by one. This provides
\begin{equation}
\begin{aligned}
\|\projBC{M}r(t)-r(t)\|_{\spaceXV}^2 \leq  &\left(\diffBC{M}\right)^2\|\momE{M}(r(t))\|_{L^2(\Omega;\mbb R^{\numEven(M)})}^2+ 2\sum_{\beta\in \{e,o\}}\sum_{q = M+1}^{\infty}\|\lambda^{\beta}_q(r(t))\|_{L^2(\Omega;\mbb R^{n_{\beta}(q)})}^2\\
\leq &\left(\diffBC{M}\right)^2\|\momE{M}(r(t))\|_{L^2(\Omega;\mbb R^{\numEven(M)})}^2\\
&+ 2\sum_{\beta\in \{e,o\}}\sum_{q = M+1}^{\infty}\frac{(2q+d)^{2k^{\beta}}}{(2(M+1)+d)^{2k^{\beta}}}\|\lambda^{\beta}_{q}(r(t))\|_{L^2(\Omega;\mbb R^{n_{\beta}(q)})}^2\\ \label{bound Pir r}
\leq &\left(\diffBC{M}\right)^2\|\momE{M}(r(t))\|_{L^2(\Omega;\mbb R^{\numEven(M)})}^2\\
&+2\sum_{\beta\in \{e,o\}}\frac{1}{(2(M+1)+d)^{2k^{\beta}}}\|r^{\beta}(t)\|_{L^2(\Omega;W_H^{k^{\beta}}(\mbb R^d))}^2,
\end{aligned}
\end{equation}
where for the last inequality we use the definition
\begin{gather*}
\|r^e(t)\|_{L^2(\Omega;W_H^{k^e}(\mbb R^d))}^2 = \sum_{q = 0}^{\infty}(2q+d)^{2k^e}\|\momE{q}(r(t))\|_{{L^2(\Omega;\mbb R^{\numOdd(q)})}}^2.
\end{gather*}
Above relation proves the bound for $\|\projBC{M}r-r\|_{\spaceXV}$. To prove the convergence rate we use the last inequality in \eqref{bound Pir r}. The convergence rate of terms involving $\|r^{e/o}(t)\|_{L^2(\Omega;W_H^{k^{e/o}}(\mbb R^d))}$ follows trivially, and to obtain a convergence rate for the term involving $\diffBC{M}$ we use
the definition of $R^{(M)}$ to find

\begin{equation*}
\begin{aligned}
\left(\diffBC{M}\right)^2\|\momE{M}(r)\|_{C^0([0,T];L^2(\Omega;\mbb R^{\numEven(M)}))}^2 =  &\|R^{(M)}\MatFluxSm{M}{M}-\MatHalfSm{M}{M}\|_2^2\|\momE{M}(r)\|_{C^0([0,T];L^2(\Omega;\mbb R^{\numEven(M)}))}^2 
\end{aligned}
\end{equation*}

\begin{equation}
\begin{aligned}
 \leq &\left(\|\left(\MatFluxBi{M}{M-1}\right)^{-1}\MatFluxSm{M}{M}\|_2+\|\MatHalfSm{M}{M}\|_2\right)^2\|\momE{M}(r)\|_{C^0([0,T];L^2(\Omega;\mbb R^{\numEven(M)}))}^2\\
  \leq &\frac{C}{M^{2k^e-1}}. \label{rate first term}
\end{aligned}
\end{equation}
The last inequality in the above relation follows from the matrix norm bound given in \autoref{normB} and from the following estimate
\begin{equation}
\begin{aligned}
\|\momE{M}(r(t))\|^2_{L^2(\Omega;\mbb R^{\numEven(M)})} \leq &\sum_{m=M}^{\infty}\|\momE{m}(r(t))\|^2_{L^2(\Omega;\mbb R^{\numEven(M)})} \leq \sum_{m=M}^{\infty}\left(\frac{2m+d}{2M+d}\right)^{2k^e} \|\momE{m}(r(t))\|^2_{L^2(\Omega;\mbb R^{\numEven(M)})} \\
\leq &\frac{1}{\left(2M+d\right)^{2k^e}} \|r(t)\|^2_{L^2(\Omega;W_H^{k^e}(\mbb R^d))}. \label{pessimistic bound}
\end{aligned}
\end{equation}

In a similar way, we prove the bound and the convergence rate for $\|\projOrtho{M}r-\projBC{M}r\|_{C^0([0,T];L^2(\Omega\times\mbb R^d)}$.
Using the definition of $\projOrtho{M}$ and $\projBC{M}$ from \autoref{def projector} we find
\begin{align*}
\projBC{M}r-\projOrtho{M} r = \left((R^{(M)}\MatFluxSm{M}{M}-\MatHalfSm{M}{M})\momE{M}(r)\right)\cdot \basisO{M}\sqrt{f_0} -  \sum_{q = M+1}^{\infty}\left(\MatHalfSm{M}{q}\momE{q}(r)\right)\cdot \basisO{M}(\xi)\sqrt{f_0}
\end{align*} 
which implies 
\begin{align*}
\|\projBC{M}r(t)-\projOrtho{M} r(t)\|_{\spaceXV}^2 \leq  &\left(\diffBC{M}\right)^2\|\momE{M}(r(t))\|_{L^2(\Omega;\mbb R^{\numEven(M)})}^2  + \sum_{q = M+1}^{\infty}\|\momE{q}(r(t))\|_{L^2(\Omega;\mbb R^{\numEven(q)})}^2. 
\end{align*}
Above inequality is the same as the first inequality in \eqref{bound Pir r} but without any contribution from the odd moments of degree higher than $M$. Therefore, we get 
the bound for $\|\projBC{M}r-\projOrtho{M} r\|_{\spaceXV}^2$ and its corresponding convergence rate from \eqref{bound Pir r} and \eqref{rate first term} by removing contribution from the odd moments of order higher than $M$. 
\end{proof}
\noindent

Using the result from \autoref{convg projector} in the upper bound for $E_M$ \eqref{estimate EM} proves the error bound given in \autoref{main result}.
To arrive at the convergence rate given in \autoref{main result}, first we split the bound for the \textit{closure error} in \autoref{main result} as 
\begin{equation}
\begin{aligned}
\sqrt{(M+1)}\|\mom{M+1}(\pd_{x_i}f)\|_{C^0([0,T];L^2(\Omega;\mbb R^{n{(M+1)}}))} \leq \sqrt{(M+1)}\left(\|\momO{M+1}(\pd_{x_i}f)\|_{C^0([0,T];L^2(\Omega;\mbb R^{\numOdd(M+1)}))} \right.
\\\left. + \|\momE{M+1}(\pd_{x_i}f)\|_{C^0([0,T];L^2(\Omega;\mbb R^{\numEven(M+1)}))}\right),
\end{aligned}
\end{equation}
which results from acknowledging that $\mom{M+1}(\pd_{x_i}f) = \left(\momO{M+1}(\pd_{x_i}f)',\momE{M+1}(\pd_{x_i}f)'\right)$.
The bound for the individual moments of $r\in L^2(\Omega;W_H^k(\mbb R^d))$ in terms of $\| r\|_{L^2(\Omega;W_H^k(\mbb R^d))}$ (see \eqref{pessimistic bound}) implies that, with respect to $M$, the \textit{closure error} decays as $\mcal O(\opn{min}\{k_x^e-\frac{1}{2},k_x^o-\frac{1}{2}\}).$ The convergence rate for all the other terms in the error bound for $E_M$ follows from the fact that $\|\MatFluxBi{M}{M}\|_2\leq C \sqrt{M}$ and from the convergence rate of the projection error. 
\subsection{Sharper Estimate}

As already noted in \cite{Martin}, a bound for the individual moments of $r\in L^2(\Omega; W_H^k(\mbb R^d))$ in terms of $\|r\|_{L^2(\Omega;W_H^k(\mbb R^d))}$ is pessimistic; see the relation in \eqref{pessimistic bound}. Therefore, one can make the error bound in \autoref{main result} sharper by additionally assuming that the individual moments decay at a certain rate. The following result provides such a sharpened error bound, which is useful during numerical experiments because solutions to most numerical experiments have moments that decay at a certain rate \cite{Torrilhon2015,Martin}.

\begin{theorem} \label{sharper main result}
In addition to \autoref{regularity assumption}, assume that 
\begin{gather}
\|\lambda^{\beta}_m(f)\|_{C^0([0,T];L^2(\Omega;\mbb R^{n_{\beta}}))} < \frac{C}{m^{k^{\beta} + \frac{1}{2}}},\hspB
\|\lambda^{\beta}_m(\pd_{t}f)\|_{C^0([0,T];L^2(\Omega;\mbb R^{n_{\beta}}))} < \frac{C}{m^{k_{t}^{\beta} + \frac{1}{2}}}, \label{assume decay}\\
\|\lambda^{\beta}_m(\pd_{x_i}f)\|_{C^0([0,T];L^2(\Omega;\mbb R^{n_{\beta}}))} < \frac{C}{m^{k_{x}^{\beta} + \frac{1}{2}}},\hspB \forall\hsp i\in\{1,\dots,d\},
\end{gather}
where $\beta\in \{e,o\}$. Then, we can sharpen the convergence rate presented in \autoref{main result} to
\begin{gather}
 \omega_{\opn{shp}} = \min\left\{k^{e/o}, k_t^{e/o},k_{x}^{e/o}-\frac{1}{2}\right\}. 
\end{gather}
\end{theorem}
\begin{proof}
The result trivially follows from the above analysis by using the assumed moment decay rate \eqref{assume decay} instead of the pessimistic bound in \eqref{pessimistic bound}.
\end{proof}
\begin{remark}
Note that the Hermite-Sobolev index in $W_H^k(\mbb R^d)$,  i.e. $k$, does not provide a decay rate for individual moments. However, if moments decay at a certain rate, i.e., if $\|\lambda_m(r)\|_{L^2(\Omega;\mbb R^{n(m)}))} \leq \frac{C}{m^s}$ then $r\in L^2(\Omega;W_H^{k}(\mbb R^d)$ for $k < s-\frac{1}{2}$. A detailed discussion can be found on page 12 of \cite{Martin}.
\end{remark}

\subsection{Uniform in Knudsen-number estimate}\label{uni Knud}
 Here we are interested in the small Knudsen number regime and, in particular, we assume $\|Q\|>0$. For convenience we define the semi-norm 
 \begin{gather}
    |f|_Q := -\lan f,Q (f) \ran_{\spaceXV}, \label{semi norm Q} 
 \end{gather}
which is well-defined because of \autoref{assumption Q}.
We show that by treating the bound in \eqref{aya} differently, we get a bound for $\|e_M(t)\|_{L^2(\Omega\times\mbb R^d)}$ that scales with $\sqrt{\|Q\|}$, which (for small Knudsen numbers) is better than the scaling of $\|Q\|$ considered in \autoref{main result}. Moreover, we derive a uniform-in-Knudsen-number bound for the part of the error that is orthogonal to the null-space of $Q$. Precisely, for any function $f$ the semi-norm $|f|_Q$ scales with $\text{Kn}^{-1}$
by definition and  we derive a linear-in-$\text{Kn}^{-1}$-number bound for $|e_M|_Q$. Recall that the Knudsen number results from the non-dimensionalisation of the kinetic equation and is explicitly given below in \eqref{linearised BE}.

From \eqref{aya} we can infer 
\begin{equation}\label{aya1}
 \frac{d}{dt} \| e_M(t)\|_{\spaceXV}^2  + |e_M(t)|_Q^2 \\
 \leq (\bar A_1(t) + \bar A_3(t)) ||e_M(t)||_{L^2(\Omega\times\mbb R^d)} + \| (-Q)^{\tfrac12} \| \bar A_2(t) |e_M(t)|_Q 
\end{equation}
with
\begin{equation*}
 \begin{aligned}
  \bar A_1(t) &:= \| \Pi_M \partial_t f(t) - \hat \Pi_M \partial_t f(t)\|_{\spaceXV}, \\
  \bar A_2(t) &:=  \|f(t) - \hat \Pi_M f(t) \|_{\spaceXV},\\
  \bar A_3(t) &:= \sum_i \| \Pi_M (\xi_i (\partial_{x_i} f(t) - \hat \Pi_M \partial_{x_i} f(t))) \|_{\spaceXV},
 \end{aligned}
\end{equation*}
where we have used that $Q$ is self-adjoint and negative semi-definite, so that $-Q$ admits a square root.
The discussion in equations \eqref{bound 0} - \eqref{bound flux 3} and \autoref{convg projector} shows that for all $t \in [0,T]$ and $i\in \{1,2,3\}$, we have
\begin{equation}
 \bar A_i(t) \leq \td A_i(T) \leq A_i(T),
\end{equation}
such that we infer that 
\begin{multline}\label{aya2}
 \frac{d}{dt} \| e_M(t)\|_{\spaceXV}^2  + \frac12 |e_M(t)|_Q^2
 \leq ( A_1(T)+ \| Q \|^{\tfrac{1}{2}}A_2(T) +  A_3(T)) ||e_M(t)||_{\spaceXV} + \| Q \|  A_2(T)^2 \\
 \leq \sqrt{2\left(( A_1(T)+ \| Q \|^{\tfrac{1}{2}}A_2(T)  +  A_3(T))^2 ||e_M(t)||_{\spaceXV}^2 + \| Q \|^2  A_2(T)^4\right)}.
\end{multline}
Thus, for all $t \in [0,T]$,  
$\| e_M(t)\|_{\spaceXV}^2$ is bounded by $ z(t)$ where $z$ solves
\begin{equation}
 \frac{d}{dt} z(t)  
  = \sqrt{2\left((A_1(T)+ \| Q \|^{\tfrac{1}{2}}A_2(T) +  A_3(T) )^2 z(t) + \| Q \|^2  A_2(T)^4\right)}
\end{equation}
with $z(0)= \| e_M(0)\|_{\spaceXV}^2=0$.
The solution $z$ satisfies
\begin{multline}
 \sqrt{ (A_1(T)+ \| Q \|^{\tfrac{1}{2}}A_2(T) +  A_3(T))^2 z(t)  +  \| Q \|^2  A_2(T)^4} \\=
\frac{1}{\sqrt{2}}(A_1(T)+ \| Q \|^{\tfrac{1}{2}}A_2(T) +  A_3(T) )^2 t + 
   \| Q \|  A_2(T)^2.
\end{multline}
The above relation provides
\begin{multline}
  (A_1(T)+ \| Q \|^{\tfrac{1}{2}}A_2(T) +  A_3(T))^2 z(t)   \\ \leq
(A_1(T)+ \| Q \|^{\tfrac{1}{2}}A_2(T) +  A_3(T) )^4 t^2 + 
   \| Q \|^2  A_2(T)^4,
\end{multline}
which results in
\begin{equation}
 \sup_{t \in [0,T]} \| e_M(t)\|_{\spaceXV}^2 \leq z(T)\leq (A_1(T)+ \| Q \|^{\tfrac{1}{2}}A_2(T) +  A_3(T) )^2 T^2 + 
   \| Q \|  A_2(T)^2,
\end{equation}
and 
\begin{equation}\label{om}
 \sup_{t \in [0,T]} \| e_M(t)\|_{\spaceXV}\leq \sqrt{z(T)}\leq (A_1(T)+ \| Q \|^{\tfrac{1}{2}}A_2(T) +  A_3(T) ) T + 
   \| Q \|^{\tfrac{1}{2}}  A_2(T)=:B(T).
\end{equation}
It is worthwhile to note that the decay of $B(T)$ with respect to $M$ is the same as the decay of the bound derived in \autoref{main result}. Moreover, both the above bound and the bound in \autoref{main result} are linear in time. However, while the bound in \autoref{main result} scaled (for small Knudsen numbers) with $\|Q\|$, the bound in \eqref{om} scales with $\|Q\|^{\tfrac12}$.
In order to obtain a uniform-in-Knudsen bound for $|e_M(t)|_Q$, we return to \eqref{aya1} and integrate on $[0,T]$. This leads to
\begin{theorem}
\begin{equation}
\begin{aligned}\label{shanti}
 \int_0^T \frac12 |e_M(t)|_Q^2 dt
 \leq & \int_0^T\left( ( A_1(T) + A_3(T))  ||e_M(t)||_{\spaceXV} + \| Q \|  A_2(T)^2\right) dt,\\
 \leq & T\cdot \left(( A_1(T)+  A_3(T))  B(T) + \| Q \|  A_2(T)^2 \right),
\end{aligned}
\end{equation}
where $|\cdot |_Q$ is as defined in \eqref{semi norm Q}, $A_1,A_2$ and $A_3$ are as defined in \eqref{def A1}-\eqref{def A3}, and $B$ is as defined in \eqref{om}.
\end{theorem}
\noindent
We note the following for the above result:
\begin{enumerate}
    \item the right hand side in \eqref{shanti} is a bound for the square of the error and it decays with twice the rate of the right hand side in \autoref{main result};
    \item both sides of \eqref{shanti} scale with $\|Q\|$, i.e., it provides a uniform-in-Knudsen-number bound. It must be noted that $|e_M(t)|_Q$ is a semi-norm and it does not quantify the part of $e_M(t)$ that is in the null-space of $Q$.
\end{enumerate}
 
\subsection{Discussion}
\subsubsection{Improved Boundary Conditions:}
Model for the matrix $\OnsagerMat{M}$ (see \eqref{Onsager Mat}) is not unique and can be altered to enhance the accuracy of a moment approximation. For example, in \cite{Rana2016} authors did such alteration for the R-13 moment equations using a data-driven approach. However, due to the absence of an error bound they did not analyse the correlation between the matrix $\OnsagerMat{M}$ and the R-13 moment approximation error.

With the error bound of the projection error, we develop some insight into the extent to which the matrix $\OnsagerMat{M}$ influences the convergence rate of a moment approximation. Consider the bound for the \textit{projection error} given in \autoref{convg projector}. We decompose this bound into two parts:
\begin{gather*}
\td a= \sum_{\beta\in \{e,o\}}\frac{1}{(2(M+1)+d)^{2k^{\beta}}}\|r^{\beta}\|^2_{L^2(\Omega;W_H^{k^{\beta}}(\mbb R^d))}\hsp\text{and}\hsp a_{\diffBC{M}} = (\diffBC{M})^2 \|\momE{M}(r)\|^2_{L^2(\Omega;\mbb R^{\numEven(M)})},
\end{gather*}
where $r^{\beta}\in L^2(\Omega;W_H^{k^{\beta}}(\mbb R^d))$ for $\beta\in \{e,o\}$, and for simplicity we consider $k^{e} = k^o = k$.
Clearly, $\td a$ is independent of $\OnsagerMat{M}$ whereas $a_{\diffBC{M}}$ is dependent upon $\diffBC{M}$ which then depends upon $\OnsagerMat{M}$.

Trivially, $\td a$ is $\mcal O(M^{-k})$ whereas, since $\diffBC{M}$ is $\mcal{O}(\sqrt{M})$, $\td a_{\diffBC{M}}$ is $\mcal O(M^{-(k-\frac{1}{2})})$. Thus if one can improve the model for $\OnsagerMat{M}$ such that $\diffBC{M}$ decays faster than $\mcal O(\sqrt{M})$ then one can obtain a moment approximation which converges faster than the one presented here. Development of such a $\OnsagerMat{M}$ is beyond our present scope and will be discussed in detail elsewhere.

\subsubsection{Sub-optimality: }
The convergence analysis presented in this paper is sub-optimal.
What we mean by optimality is twofold. Firstly, optimality means that the difference between the numerical and the exact solution decays with the same rate as the best approximation error of the exact solution. Secondly, optimality would require that no additional conditions are imposed on the exact solution.
For the case at hand, the rate of convergence of the best approximation error is the Hermite-Sobolev index.
Our analysis requires additional assumptions in the sense that not only the solution but also its derivatives need to have some Hermite-Sobolev regularity. This is a common feature of the analysis of numerical schemes for hyperbolic problems, see e.g. \cite{Egger2016,Douglas73,Douglas78}.

Recalling the convergence rate presented in \autoref{main result}, we find 
\begin{gather}
\omega = \min\left\{k^{e/o}-\frac{1}{2}, k_t^{e/o}-\frac{1}{2},k_x^e-\frac{1}{2}-\underline{\frac{1}{2}},k_x^o-\frac{1}{2}\right\}, \label{convg rate dis}
\end{gather}
where $\omega$ is sub-optimal with respect to the different Hermite-Sobolev indices i.e., with respect to the different values of $k$.
We elaborate on this particular sub-optimality and show (through an example) that it results from the velocity domain in the kinetic equation being unbounded \eqref{BE}. 
Loss of half an order in all indices is a result of the boundary stabilisation error ($\Theta_M$), which grows with $\sqrt{M}$. This error gets multiplied by $\|\MatFluxBi{M}{M}\|_2$, which grows with $\sqrt{M}$, and results in a sub-optimality of an extra half appearing in the contribution from spatial derivatives; see the terms involving $A_3$ in \autoref{main result}. 

Growth in $\|\MatFluxBi{M}{M}\|_2$, which also causes the growth in $\Theta_M$, is a result of the recursion relation of Hermite polynomials \eqref{recursion} which states that the product of $\xi$ with a $M$-th order Hermite polynomial equals a linear combination of a $(M-1)$-th and a $(M+1)$-th order Hermite polynomial but with factors which grow with $\sqrt{M}$. This growth results in the coefficients of $\MatFluxBi{M}{M}$ growing as $\mcal O(\sqrt{M})$, which leads to a growth in the norm of $\MatFluxBi{M}{M}$. See appendix-\autoref{properties Aoe} and appendix-\autoref{norm mat and op} for details of the structure of $\MatFluxBi{M}{M}$ and $\Theta_M$, respectively. 
The use of Hermite polynomials as basis functions (and thus the growth in $\|\MatFluxBi{M}{M}\|_2$) is related to the velocity domain of the kinetic equation \eqref{BE} being unbounded. For kinetic equations with a bounded velocity space, it might be possible to have basis functions such that $\|\MatFluxBi{M}{M}\|_2$ does not grow with $M$, which would remove the additional sub-optimality in the Hermite-Sobolev indices of the spatial derivatives. As an example, consider the radiation transport equation for which the velocity space is a unit sphere and is thus bounded. A moment approximation can, therefore, be developed with the help of spherical harmonics and contrary to Hermite polynomials, the recursion relation of spherical harmonics is such that $\|\MatFluxBi{M}{M}\|_2\to 1$ as $M\to\infty$  \cite{Martin,PNIntro,Egger2016}. \autoref{compare norm} shows a comparison between the norm of $\MatFluxBi{M}{M}$ for a $\mbb S^2$ and a $\mbb R^3$ velocity domain. Clearly, as $M$ is increased, for a $\mbb S^2$ velocity space $\|\MatFluxBi{M}{M}\|_2$ approaches its limiting value of one whereas
for a $\mbb R^3$ velocity space $\|\MatFluxBi{M}{M}\|_2$ grows with $\mcal O(\sqrt{M})$. Thus for radiation transport, owing to the boundedness of $\|\MatFluxBi{M}{M}\|_2$ with $M$, we expect that one can entirely remove the second type of sub-optimality present in $\omega$, i.e., one can get a convergence rate which is the same as the Hermite-Sobolev indices. Such a result would be in agreement with the error estimates presented in \cite{Egger2016,Martin}.
\begin{figure}[ht!]
\centering
\subfigure {
\includegraphics[width=3in]{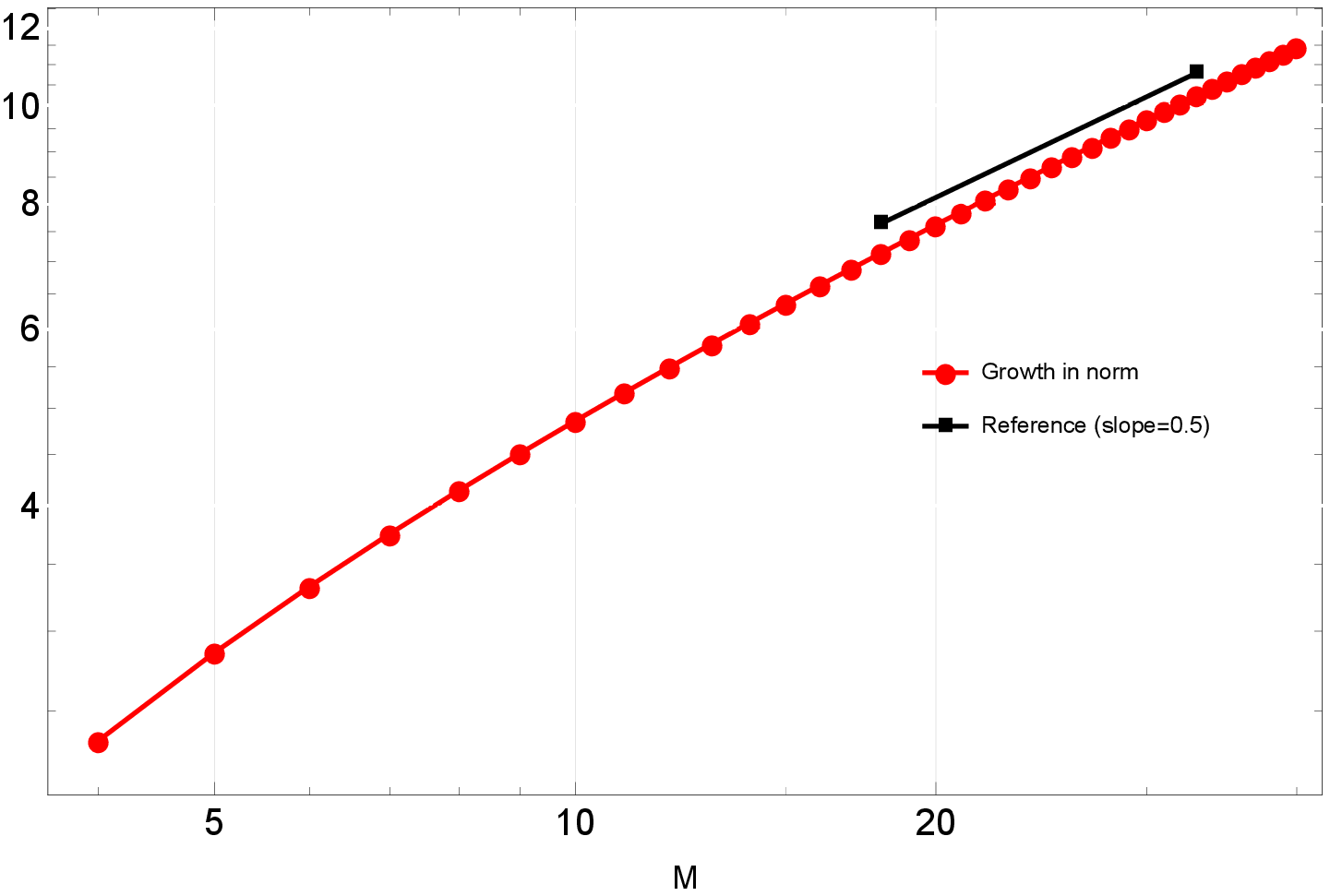} }
\hfill
\subfigure{
\includegraphics[width=3in]{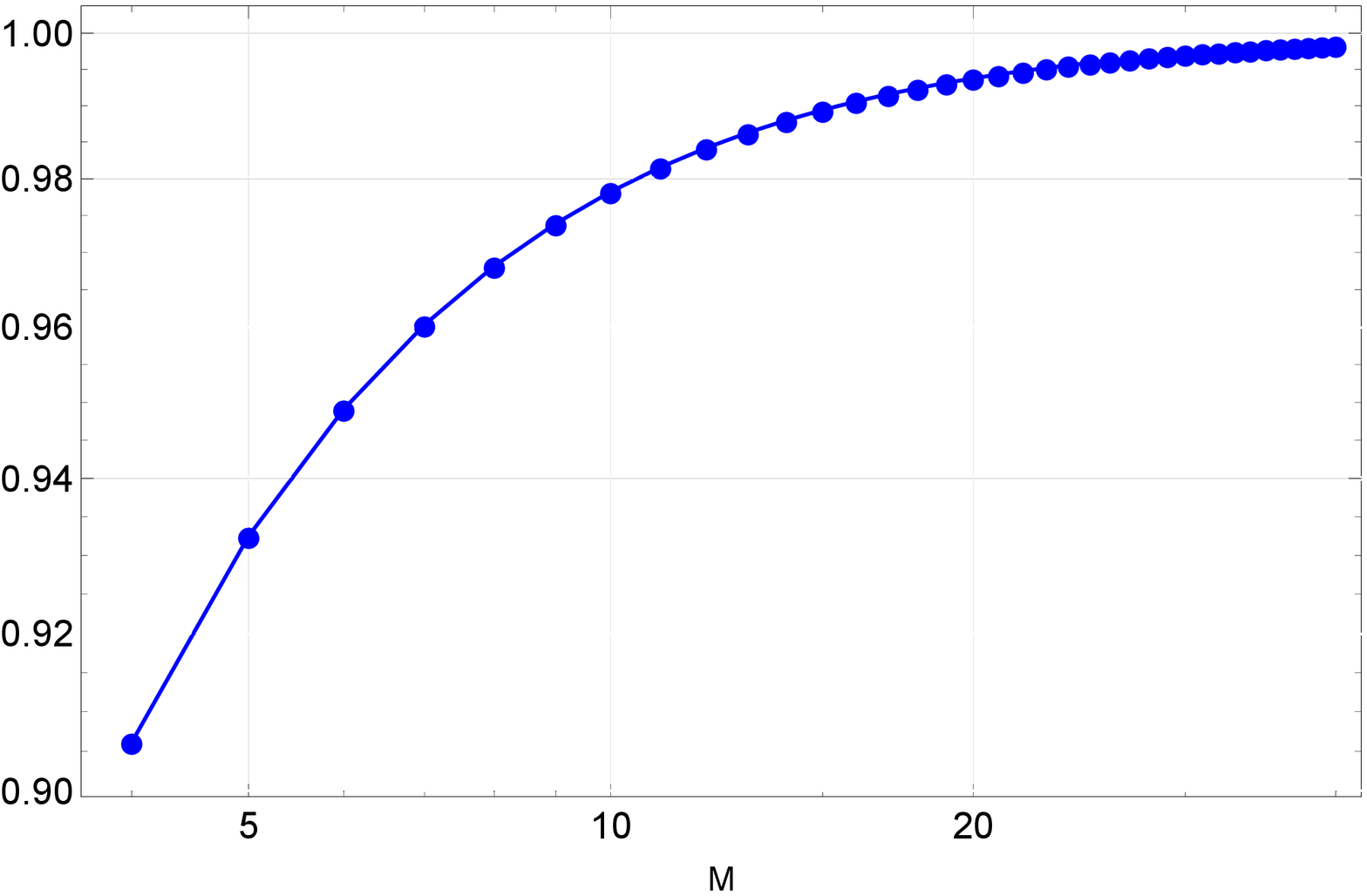}}	
\hfill
\caption{ growth in $\|\MatFluxBi{M}{M}\|_2$ with $M$ for: (i) left, $\mbb R^3$ velocity space and (ii) right, $\mbb S^2$ velocity space.}\label{compare norm}
\end{figure} 

\section{Examples: Linearised Boltzmann and BGK equations} \label{an example}
We give examples of kinetic equations which fall into the framework presented above. In particular, we discuss the conditions under which the linearised Boltzmann and the linearised BGK equation fall into our framework.  

With $\bar f:D\to \mbb R^+$, $(t,x,\xi)\mapsto \bar f(t,x,\xi)$, we denote the phase density function of a gas and we normalise $\bar f$ such that
the density ($\bar{\rho}$), the mean flow velocity ($\bar{v}$), and the temperature in energy units ($\bar{\theta}$) of the gas are given as: 
$
\bar{ \rho} = \int_{\mbb{R}^d} \bar f d\xi$, $
\bar{\rho} \bar{v} = \int_{\mbb{R}^d} \xi\bar f d\xi$, $
\bar{\rho} \bar{v}\cdot \bar{v} + d\bar{\rho} \bar{\theta} =\int_{\mbb{R}^d}\xi\cdot \xi\bar f d\xi.$
For convenience, we non-dimensionalise all quantities with some reference density $\rho_0$, temperature $\theta_0$ and length scale $L$.
The evolution of $\bar f$ is governed by the non-linear kinetic equation given as \cite{Struchtrupbook}
\begin{gather}
(1,\xi)\cdot \nabla_{(t,x)} \bar f = \frac{1}{\opn{Kn}}\bar Q(\bar f,\bar f), \label{non-linear BE}
\end{gather}
where $\opn{Kn}$ is the so-called Knudsen number which results from non-dimensionalisation, and $\bar Q$ is a non-linear collision operator. We consider $\bar Q$ to be either the Boltzmann or the BGK collision operator given as 
\begin{equation*}
\begin{aligned}
\text{Boltzmann Operator: } &\bar Q_{\opn{BE}}(\bar f,\bar f) = \int_{\mbb R^d}\int_{\mbb S^{d-1}}\mcal B(\xi-\xi_*,\kappa) \left(f(\xi^')f_0(\xi^'_*)-f(\xi)f_0(\xi_*)\right)d\kappa d\xi_*;\\
\text{BGK Operator: } &\bar Q_{\opn{BGK}}(\bar f,\bar f) = (\bar f_{\mcal M} - \bar f).
\end{aligned}
\end{equation*}
Above, the velocities $\xi_*^'$ and $\xi^'$ are post-collisional and result from the pre-collisional velocities $\xi_*$ and $\xi$.
The collision kernel ($\mcal B$) depends on the interaction potential between the gas molecules and is non-negative by physical assumptions. Moreover, $\bar f_{\mcal M}$ is a Maxwell-Boltzmann distribution function given as 
\begin{gather*}
\bar f_{\mcal M}(\xi ; \bar{\rho},\bar{v},\bar{\theta}) = \frac{\bar{\rho}}{\sqrt[d]{2\pi\bar{\theta}}}\exp\left[-\frac{(\xi-\bar{v})\cdot (\xi-\bar{v})}{2\bar{\theta}}\right].
\end{gather*}

For low Mach number flows, we assume $\bar f$ to be a small perturbation of a ground state $f_0 = \bar f_{\mcal M}(\xi ; \rho_{0},0,\theta_0)$, i.e. $\bar f = f_0 + \epsilon\sqrt{f_0} f$, where $\epsilon$ is some smallness parameter. Substituting the linearisation into the non-linear kinetic equation \eqref{non-linear BE} and considering only $\mcal O(\epsilon)$ terms, we find the evolution equation for $f$
\begin{gather}
(1,\xi)\cdot \nabla_{(t,x)}f = \frac{1}{\opn{Kn}} Q(f), \label{linearised BE}
\end{gather}
where $Q$ is the linearisation of $\bar Q_{\opn{BE}/\opn{BGK}}$ about $f_0$ and is given as 
\begin{equation*}
\begin{aligned}
\text{Linearised Boltzmann Operator: } &Q_{\opn{BE}}(\bar f) = \int_{\mbb R^d}\int_{\mbb S^{d-1}}\mcal B(\xi-\xi_*,\kappa)
\sqrt{f_0(\xi_*)f_0(\xi)}\\
&\hspace{1cm} \left(\frac{f(\xi^')}{\sqrt{f_0(\xi^')}}+\frac{f(\xi_*^')}{\sqrt{f_0(\xi_*^')}}-\frac{ f(\xi_*)}{\sqrt{f_0(\xi_*)}}-\frac{f(\xi)}{\sqrt{f_0(\xi)}}\right)d\kappa d\xi_*;\\
\text{Linearised BGK Operator: } &Q_{\opn{BGK}}(f) = (f_{\mcal M} - \bar f).
\end{aligned}
\end{equation*}
Above, $f_{\mcal M}\sqrt{f_0}$ is a linearisation of $\bar f_{\mcal M}$ about $f_0$ and is given as 
\begin{align}
f_{\mcal M}(\xi;\rho,v,\theta) := \left(\rho + v\cdot \xi + \frac{\theta}{2}\left(\xi\cdot \xi-3\right)\right)\sqrt{f_0(\xi)}, \label{equilibrium distribution}
\end{align}
where $\rho$, $v$ and $\theta$ are deviations of $\bar{\rho}$, $\bar{v}$ and $\bar{\theta}$ from their respective ground states.

We discuss whether the collision operators $Q_{\opn{BE}/\opn{BGK}}$ satisfy \autoref{assumption Q}. One can show that both $Q_{\opn{BE}/\opn{BGK}}$ are negative semi-definite and self-adjoint, and that $Q_{\opn{BGK}}$ is bounded on $L^2(\mbb R^d)$; see \cite{Carlos} for details. Thus $Q_{\opn{BGK}}$ satisfies \autoref{assumption Q}. Below in \autoref{bounded Q} we summarise the assumptions that make $Q_{\opn{BE}}$ a bounded operator, which results in $Q_{\opn{BE}}$ satisfying \autoref{assumption Q}. 

As compared to the general kinetic equation \eqref{BE}, our example of the linearised Boltzmann (or the BGK) equation \eqref{linearised BE} has an additional factor of $1/\opn{Kn}$, which scales the collision operator. From the bound on $\|e_M(t)\|_{\spaceXV}$ (in \eqref{om}) we find that such a scaling introduces a factor of $1/\sqrt{\opn{Kn}}$ in front of the term $\|Q\|^{\tfrac{1}{2}}A_2(T)$ appearing in the error bound.
An asymptotic analysis in terms of the Knudsen number can tell us how the error bound (or equivalently $A_2(T)$) behaves as the Knudsen number is chosen smaller and smaller. Authors in \cite{GradAsymp} conduct such an analysis for initial value problems. For initial boundary value problems, an asymptotic analysis is available only for the simplified Broadwell equation \cite{BCLayerBroadwell}. We hope to cover the asymptotic study of the error bound in our future work. 
Although the bound on $\|e_M\|_{\spaceXV}$ is sub-optimal in $\opn{Kn}$, the bound on $|e_M|_{Q}$ (given in \eqref{shanti}) is uniform in $\opn{Kn}$. However, the semi-norm $|e_M|_{Q}$ only quantifies the part of the error that is orthogonal to the null-space of $Q$, and it is unclear how to get a uniform in $\opn{Kn}$ bound for the error in the null-space of $Q$.

\begin{remark}\label{bounded Q}

Assume that we can split $Q_{\opn{BE}}$ as
\begin{align}
Q_{\opn{BE}}(f)(\xi) = \td Q(f)(\xi)-v(\xi)f(\xi),\hspB v(\xi) = \int_{\mbb R^3}\int_{\mbb S^2}\mcal B(\xi-\xi_*,\kappa) \sqrt{f_0(\xi_*)}d\kappa d\xi_*, \label{collision frequency}
\end{align}
where $v(\xi)\geq 0$ is the collision frequency and $\td Q$ is the remaining integral operator. The explicit form of $\td Q$ can be found in \cite{CutOffQ}. We can bound $Q$ on $\spaceV$ by bounding $\td Q$ and $v(\xi)$ on $\spaceV$ and $\mbb R^+$, respectively. 

We discuss assumptions that allow for the above splitting of $Q$, and for a bound on $\td Q$ and $v(\xi)$. Details related to our assumptions can be found in \cite{GradRGD,CutOffQ,Carlos}. Assuming an inverse power law potential, we express $\mcal B(\xi-\xi_*,\kappa)$ as 
\begin{gather*}   
\mcal B(\xi-\xi_*,\kappa) = \Psi(|\xi - \xi_*|)b(\cos \theta),\hspB \Psi(|\xi - \xi_*|) = |\xi - \xi_*|^\gamma,\hspB \gamma\in (-3,1],\hspB \cos\theta = \frac{\xi-\xi_*}{|\xi-\xi_*|}\cdot \kappa.
\end{gather*}
Assuming Grad's angular cut-off results in $\theta\mapsto b(\cos\theta)\in L^1([0,\pi])$. This makes $v(\xi)$ well-defined and allows us to split $Q$ as above \eqref{collision frequency}.
The operator $\td Q$ is bounded on $L^2(\mbb R^d)$ for $\gamma \in (-3,1]$. Moreover, $|v(\xi)|$ is bounded for all $\gamma \in (-3,0]$. Therefore, $Q_{\opn{BE}}$ is bounded on $L^2(\mbb R^d)$ for inverse power law potentials with an angular cut-off and $\gamma \in (-3,0]$.

\end{remark}

\section{Numerical Results}
Through numerical experiments, we validate the convergence rates presented in the earlier sections by comparing the observed convergence rate with the predicted one. The solution to our numerical experiment has moments that decay at a certain rate and hence we use the sharper estimate presented in \autoref{sharper main result}. With $f_{\opn{ref}}$ we denote the reference solution and we set $f_{\opn{ref}} = f_{M_{\opn{ref}}}$ with $M_{\opn{ref}}$ being sufficiently large.
To compute the observed convergence rate, which we denote by $\omega_{\opn{obs}}$, we first compute the moment approximation error through $E_M(T)=f_{\opn{ref}}(T)-f_M(T).$
Then, we compute $\omega_{\opn{obs}}$ as the slope of the linear curve that minimises the $L^2$ distance to the curve $(log(M),log(\|E_M(T)\|_{\spaceXV}))$. The predicted convergence rate, which we denote by $\omega_{\opn{pre}}$, follows from \autoref{sharper main result} and is given as 
 $$\omegaPred = \min\left\{k^{e/o}, k_t^{e/o},k_x^{e/o}-\frac{1}{2}\right\}.$$ 
To compute the different values of $k$ we first define the $L^2$ norms of the moments of $f_{\opn{ref}}$ and its derivatives 
 \begin{equation}
\begin{gathered}
N_m^{(x_i)} := \|\mom{m}(\pd_{x_i}f_{\opn{ref}})\|_{C^0([0,T];L^2(\Omega;\mbb R^{n(m)}))},\hspB N_m^{(t)} := \|\mom{m}(\pd_t f_{\opn{ref}})\|_{C^0([0,T];L^2(\Omega;\mbb R^{n(m)}))},\\ N_m := \|\mom{m}(f_{\opn{ref}})\|_{C^0([0,T];L^2(\Omega;\mbb R^{n(m)}))}. \label{def N}
\end{gathered}
\end{equation}
Let $s^o$ represent the slope of the linear curve that has the minimum $L^2$ distance to the curve\\
$(log(m),log(N_m^o))$ with $N_m^o$ being the same as $N_m$ but with a dependency on only the odd moments. We approximate $k^o$, and similarly the other $k$'s, by $k^o\approx s^o-1/2$. Once values of $k$ are known we can compute $\omega_{\opn{pre}}$ using the above expression. 
To quantify the discrepancy between the observed and the predicted convergence rates, we define
\begin{align*}
\errorOmega = \omegaObs-\omegaPred.
\end{align*}

For simplicity, we stick to a one dimensional physical and velocity space i.e., $d=1$ and $\Omega = (0,1)$.
To discretize the $1D$ physical space we use a discontinuous galerkin (DG) discretization with first-order polynomials and $500$ elements. For temporal discretization, we use a fourth-order explicit Runge-Kutta scheme. Our DG scheme is based upon a weak boundary implementation that preserves the stability of the moment approximation \eqref{stability fM} on a spatially discrete level; see \cite{Torrilhon2017} for details.
 Note that in \autoref{main result} we assumed $\Omega$ to be the half-plane but we can extend the analysis to $\Omega = (0,1)$ through the following argument.
The projection operator ($\projBC{M}$ in \autoref{def projector}) is defined with respect to the boundary conditions at $x=1$ and a similar projection operator can also be constructed for the boundary conditions at $x=0$. By taking a linear combination of the projection operation defined with respect to boundary conditions at $x=0$ and $x=1$, analogous results as those presented in \autoref{main result} (and \autoref{sharper main result}) can be obtained for $\Omega = (0,1)$. 

As initial data we consider
$
f_{I}(x,\xi) = \frac{\rho_I(x)}{\sqrt{2\pi}}\exp\left(-\frac{\xi^2}{2}\right)$ with $\rho_I(x) := \exp\left[-\left(x-0.5\right)^2\times 100\right]
$
which corresponds to a Gaussian density profile with all the higher order moments being zero. As boundary data we consider vacuum at both the ends ($x = 0$ and $x=1$) i.e., $\fIn = 0.$ As final time we consider $T = 0.3$, and we choose $M_{ref}=200$.

\autoref{decay 1D Inflow} shows the decay in the $L^2$ norm of the moments defined in \eqref{def N}, and the corresponding Hermite-Sobolev indices are given in \autoref{table: decay 1D}. The moments of the solution and its derivatives have a Hermite-Sobolev index that is close to $1.5$, which signifies that the reference solution is sufficiently regular along the velocity space. As expected, the moment approximation error decreases as the value of $M$ is increased; see \autoref{variation EM}. However, contrary to the previous results \cite{Torrilhon2015}, the convergence behaviour of the approximation error does not show any oscillations. 

\autoref{rate mom ref} shows the observed and the predicted convergence rate.
The observed approximation error converges with an order of $1.16$ and is under-predicted by a value of $0.19$. 
For the sake of validation, we also compute the convergence rates with the reference solution obtained through a discrete velocity method (DVM); see \cite{LucDVM} for details of a DVM. With DVM as the reference, we obtain
$
\omegaPred = 0.98,$ $\omegaObs = 1.15$ and $\errorOmega =\omegaObs-\omegaPred= 0.17
$
which is very similar to the results obtained with a moment reference solution \autoref{rate mom ref}. 

\begin{table}[ht!]
\begin{center}
\begin{tabular}{ c|c  }
 \hline
Quantity& Hermite-Sobolev index (= Decay Rate-$0.5$) \\
 \hline
$N_m$   & $1.8$ ($=k^e=k^o$)\\
\hline
$N_m^{(t)}$ &   $1.45$ ($=k^e_t=k^o_t$)\\ 
\hline
$N_m^{(x)}$ & $1.47 $ ($=k^e_x=k^o_x$)\\
\hline
\end{tabular}
 \caption{\textit{Hermite-Sobolev indices corresponding to the time integrated magnitude of moments defined in \eqref{def N}.} }\label{table: decay 1D}
\end{center}
\end{table}

\begin{table}[ht!]
\begin{center}
\begin{tabular}{ c|c|c|c  }
 \hline
Values of M& $\omegaPred$ & $\omegaObs$ & $\errorOmega = \omegaObs-\omegaPred$\\
 \hline
Odd  & $0.97$ & $1.16$ & $0.19$\\
\hline
Even &   $0.97$ & $1.16$ & $0.19$\\ 
\hline
\end{tabular}
 \caption{\textit{Observed and predicted convergence rates.}} \label{rate mom ref}
\end{center}
\end{table}

\begin{figure}[ht!]
\centering
\subfigure [decay of $N_m$.]{
\includegraphics[width=3in]{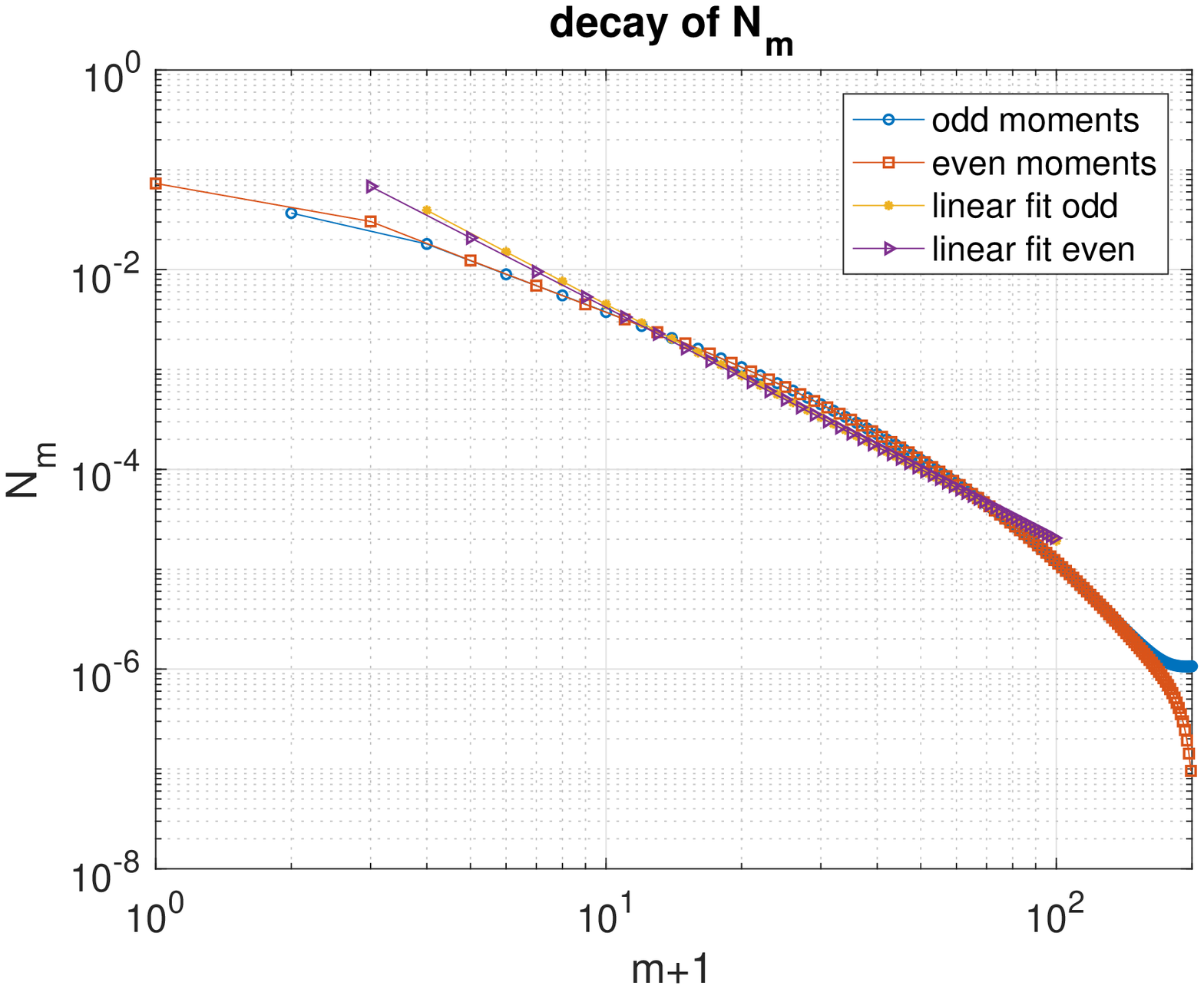} \label{NQ}}
\hfill
\subfigure [decay of $N_m^{(t)}$.]{
\includegraphics[width=3in]{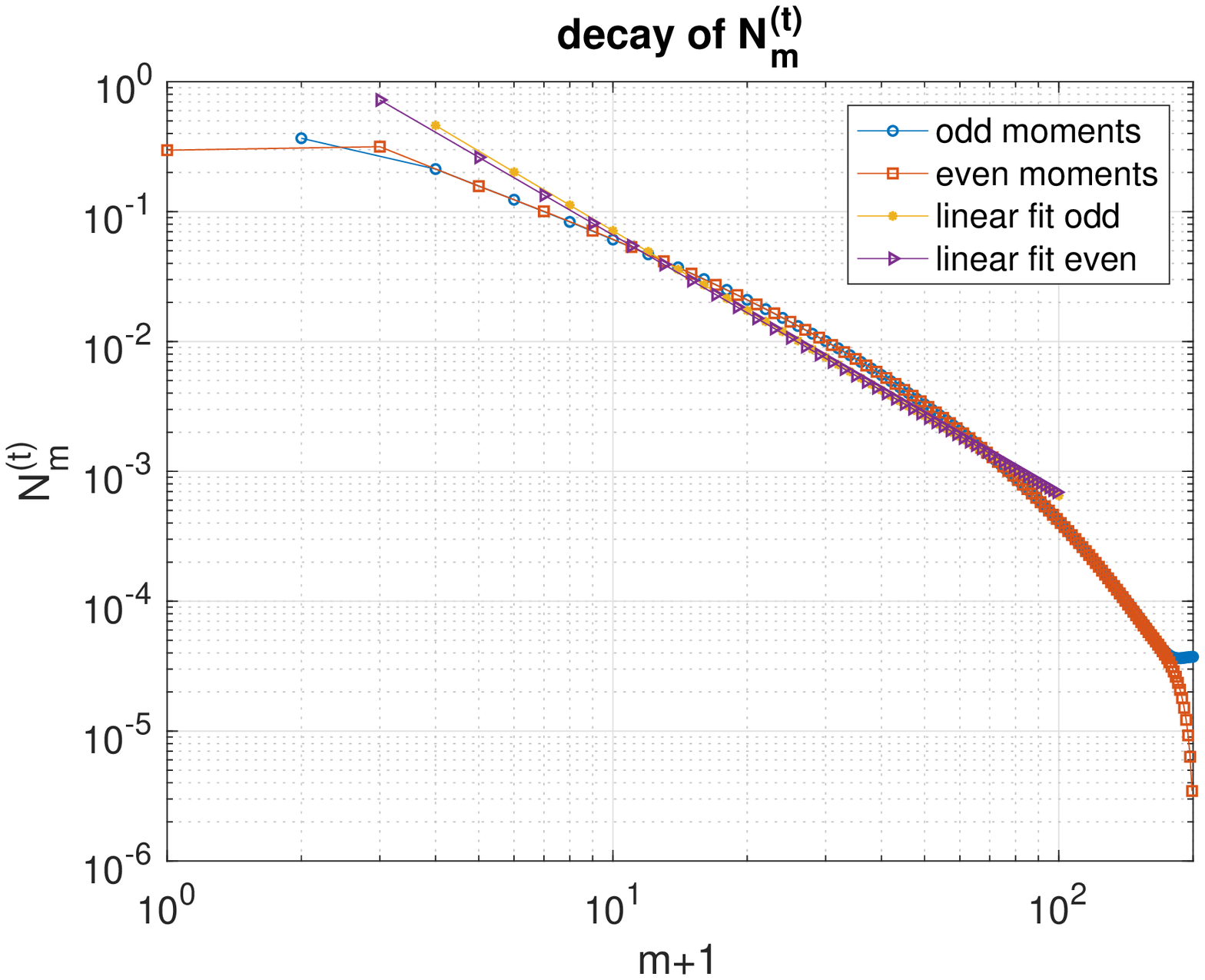} \label{NQDT}}
\hfill
\begin{center}
\subfigure [decay of $N_m^{(x)}$.]{
\includegraphics[width=3in]{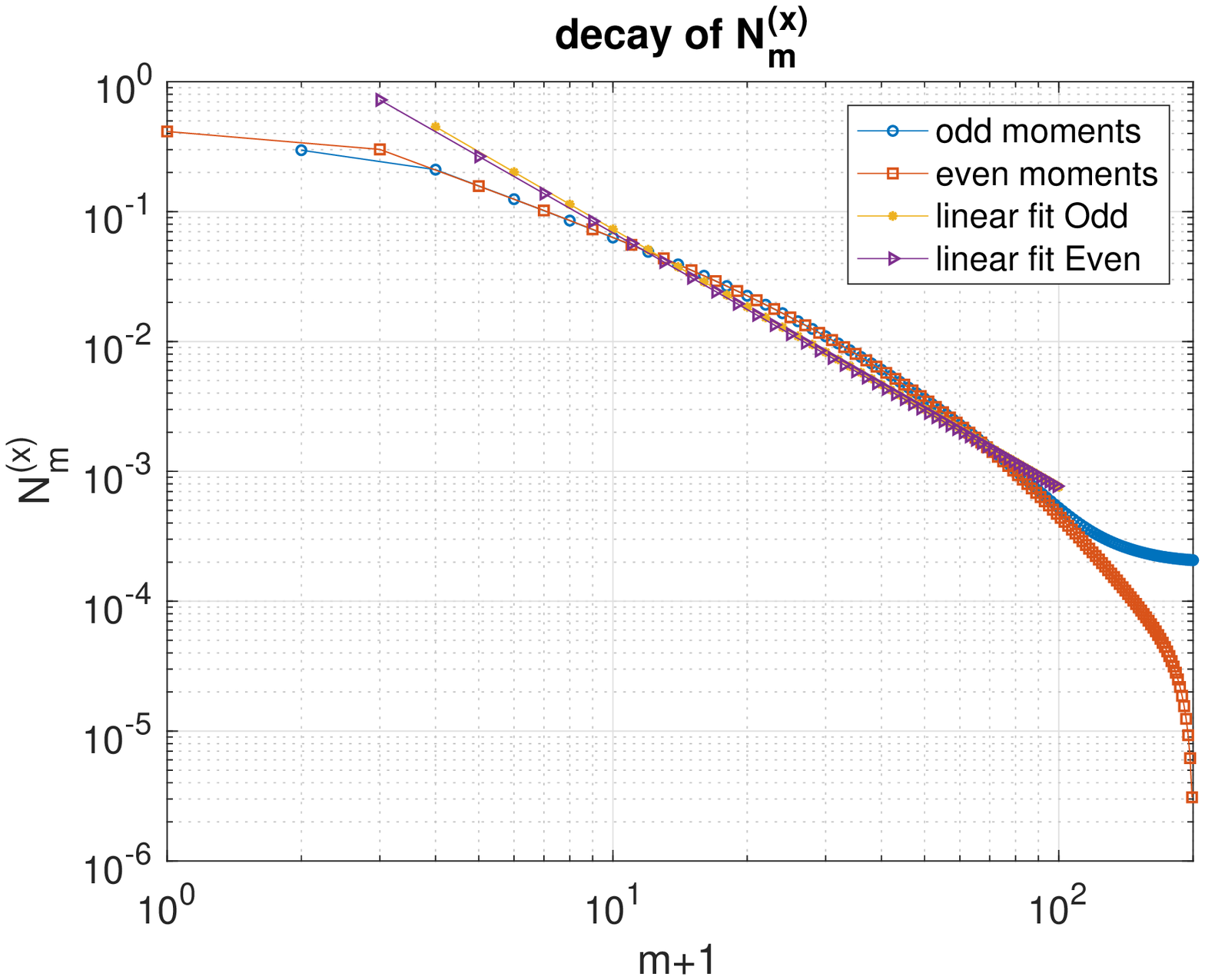} \label{NQDX}}
\end{center}
\hfill
\caption{\textit{Plots depict the decay of the various quantities, defined in \eqref{def N}, obtained through a refined moment approximation ($M=200$). All plots are on a log-log scale.}}	 \label{decay 1D Inflow}
\end{figure} 

\begin{figure}[ht!]
\begin{center}
\includegraphics[width=3in]{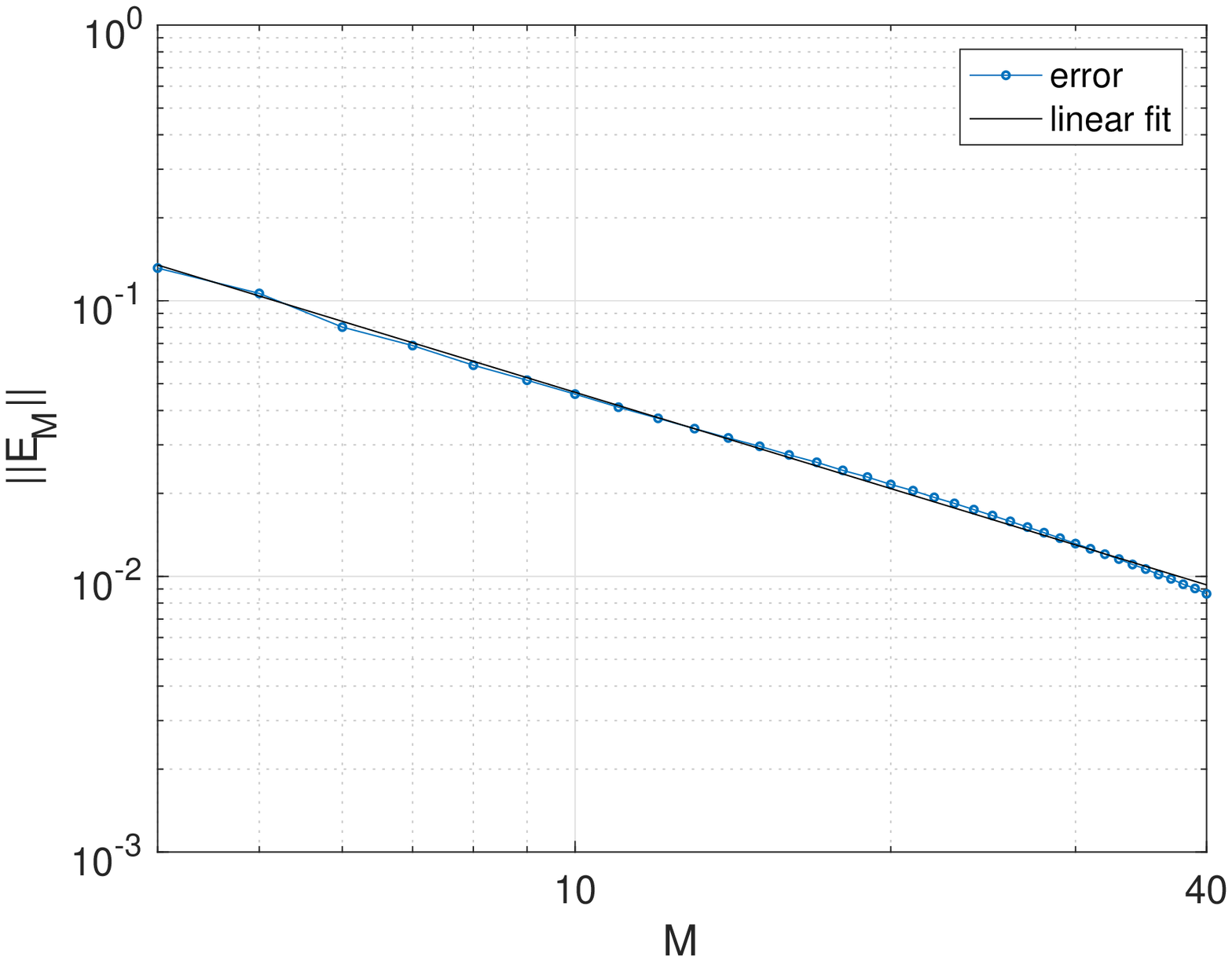} \label{EM}
\end{center}
\hfill
\caption{\textit{Decay of the approximation error, on a log-log scale, for different values of $M$.}}	 \label{variation EM}
\end{figure}

\begin{remark}
Authors in \cite{Martin} observed that moment decay rates computed using $f_{\opn{ref}}$ might show some artefacts for higher-order moments. To remove these artefacts we follow the methodology proposed in \cite{Martin}, i.e., we compute decay rates from only those values of $N_m$'s whose values computed through $M_{ref}$ and $M_{ref}-1$ differ by less than $3$ percent.
\end{remark}

\section{Conclusion}
Using a Galerkin type approach, under certain regularity assumptions on the solution, the global convergence of Grad's Hermite approximation to a linear kinetic equation was proved. 
The speed of convergence was quantified by proving convergence rate which, as was expected, depends on the velocity space Sobolev regularity of the solution.
The proposed convergence rate was found to be sub-optimal, in the sense that it is one order lower than the convergence rate of the best-approximation in the Galerkin spaces under consideration. Growth in the norm of the Jacobian corresponding to the flux of moment equations was found to be the reason for this sub-optimality.
For validation of the proven convergence rate, a numerical experiment involving the linearised BGK-equation was conducted. For a moderately high Knudsen number ($\opn{Kn} = 0.1$), the observed convergence rate matched with the predicted convergence rate with acceptable accuracy. 
 
 \section{Acknowledgements}
JG thanks the Baden-Wuerttemberg foundation for support via the project 'Numerical Methods for Multi-phase Flows with Strongly Varying Mach Numbers'. NS and MT thanks to the funding by the Deutsche Forschungsgemeinschaft (DFG, German Research 
Foundation), Project number: \\
320021702/GRK2326, Project Name: Energy, Entropy, and 
Dissipative Dynamics (EDDy).

\begin{appendices}
\section{Proof of Lemma 2.1}\label{proof lemma2.1}
By splitting the integral over $\xi_1$, we find $
\lan \CollectBasisO{M}\sqrt{f_0}, r\ran_{\spaceV} = \lan \CollectBasisO{M}\sqrt{f_0}, r\ran_{L^2(\mbb R^+\times\mbb R^{d-1})} + \frac{1}{2}\bcInhomo(r).
$ Expressing $r$ as $r = r^e + r^o$ and using $\lan\CollectBasisO{M}\sqrt{f_0}, r^e\ran_{\spaceV} = 0$ in the previous expression, we find the desired result. To derive an expression equivalent to \eqref{int ro re}, we express $r^o$ and $r^e$ as 
$
r^o = \sum_{m=1}^{\infty} \momO{m}(r)\cdot \basisO{m}\sqrt{f_0(\xi)}$ and $ r^e = \sum_{m=0}^{\infty} \momE{m}(r)\cdot \basisE{m}\sqrt{f_0(\xi)}
$ respectively
and replace these expansion in \eqref{int ro re} to find $
\CollectMomO{M}(r)  = \lim_{q\to\infty} \MatHalfBi{M}{q}\CollectMomE{q}(r) + \bcInhomo(r)$. 

We consider $\lim_{q\to\infty} \MatHalfBi{M}{q}$ to be an operator defined over $l^2$ in the sense of 
$$(\lim_{q\to\infty} \MatHalfBi{M}{q}) x:= (\lim_{q\to\infty} \MatHalfBi{M}{q} x),\hsp\forall\hsp x\in l^2.$$
We now show that $\lim_{q\to\infty} \MatHalfBi{M}{q}$ is well defined on $l^2$ which is equivalent to showing that the limit $\lim_{q\to\infty} \MatHalfBi{M}{q} x$ is well defined. Let $x\in l^2$ and let $x^q
\in\mbb{ R}^q$ be a vector containing the first $q$ elements of $x$. To extend $x^q$ by zeros, we additionally define $\bar{x}^q\in l^2$ which has the same first $q$ elements as $x$ and whose all the other elements are zero. 
From the definition of $\MatHalfBi{M}{q}$ (i.e. \autoref{def B}) we find 
$
 \MatHalfBi{M}{q}x^q = 2\lan\CollectBasisO{M}\sqrt{f_0}, g^q\ran_{L^2(\mbb R^+\times\mbb R^{d-1})}$ where $g^q = (\CollectBasisE{q}\cdot x^q) \sqrt{f_0}.
$
Trivially, $\bar{x}^q$ converges to $x$ in $l^2$. This implies that $g^q$ converges in $\spaceV$. Then, by the continuity of the inner product of $L^2(\mbb R^+\times \mbb R^{d-1})$, we have the convergence of $\MatHalfBi{M}{q}x^q$ in $\mbb R^{\numOddTot{M}}$.

 \section{Structure of $\MatFluxBi{M}{M}$} \label{properties Aoe}
We discuss in detail the structure of $\MatFluxBi{M}{M}$ which will be needed for the proof of \autoref{normB}.
From the definition of $\MatFluxBi{M}{M}$ it is clear that it contains blocks of the integral\\
$
 D^{(k,l)} = \lan \psiO{k}\sqrt{f_0},\xi_1\psiE{l}^'\sqrt{f_0}\ran_{\spaceV}$ and $ D^{(M,M+1)} = 0
$
where the second relation is a result of only considering basis functions upto degree $M$ in our moment approximation \eqref{petrov galerkin}.
Recursion of the Hermite polynomials \eqref{recursion} provides
$
 \psiO{k}\xi_1 = d^{(k,k-1)} \psiE{k-1} + d^{(k,k+1)} \hat{\psi}_{k+1}^e,
$
 where $\hat{\psi}_{k+1}^e$ is vector containing the first $\numOdd(k)$ components of $\psi_{k+1}^e$. Moreover, matrices $d^{(k,k-1)},d^{(k,k+1)}\in\mbb R^{\numOdd(k)\times \numOdd(k)}$ are diagonal matrices containing the square root entries appearing in the recursion relation. 
Using orthogonality of basis functions, we express $D^{(k,l)}$ as
  \begin{align}
D^{(k,l)} = \begin{cases}
d^{(k,k-1)} \intV{\psiE{k-1}\psiE{k-1}^'f_0} = d^{(k,k-1)},&\quad l = k-1\\ d^{(k,k+1)}\intV{\hat{\psi}_{k+1}^e\left(\psiE{k+1}\right)'f_0} = \left(\begin{array}{c c}
d^{(k,k+1)}& 0
 \end{array}\right),&\quad  l = k+1 \\
 0,&\quad \text{else} \label{simple Dkl}
 \end{cases}
 \end{align}
Note that $D^{(k,k-1)} \in \mbb R^{\numOdd(k)\times (\numEven(k-1))}$, where $\numEven(k-1) = \numOdd(k)$, whereas $D^{(k,k+1)}\in \mbb R^{\numOdd(k)\times\numEven(k+1)}$.   
Since, $n_e(k) = n_o(k+1)$, $\MatFluxBi{M}{M}$ consists of blocks of $D^{(k,k-1)}$ on its main diagonal and blocks of $D^{(k,k+1)}$ on its off diagonal with no entries below the main diagonal. 
From the recursion of Hermite polynomials \eqref{recursion}, we conclude 
\begin{gather}
d_{ii}^{(k,k-1)} = \sqrt{\left(\beta_k^{(1,o)}\right)_i},\quad d_{ii}^{(k,k+1)} = \sqrt{\left(\beta_k^{(1,o)}\right)_i+1},\quad i \in \{1,\dots,\numOdd(k)\}. \label{def d}
\end{gather}
where $\beta_k^{(1,o)}$ is as defined below
\begin{definition} \label{def betako}
Let $\beta_k^o\in \mbb R^{\numOdd(k)\times d}$ be such that each row of $\beta_k^o$ contains the multi-index of the odd basis functions contained in $\psiO{k}$. Moreover, let $\beta_k^{(1,o)}\in\mbb R^{\numOdd(k)}$ represent the first column of $\beta_k^o$.
\end{definition}

%\begin{gather}
%d_{ii}^{(k,k-1)} = 
%\begin{cases}
%\sqrt{k-2p(i)+2},\quad &k\text{ odd}\\
%\sqrt{k-2p(i)+1},\quad &k\text{ even}
%\end{cases} 
%\quad \forall i\in \{1,\dots \numOdd(k)\}\\ d_{ii}^{(k,k+1)} = 
%\begin{cases}
%\sqrt{k-2p(i)+3},\quad &k\text{ odd}\\
%\sqrt{k-2p(i)+2},\quad &k\text{ even}
%\end{cases}
%\quad \forall i\in \{1,\dots \numOdd(k)\}  \label{def D1}
%\end{gather}

Note that all the entries in $\beta_k^{(1,o)}$ are odd.
Therefore, all the entries along the diagonal of $d^{(k,k+1)}$ and $d^{(k,k-1)}$ are square roots of even and odd numbers respectively. 
 It can be shown that the number of times one appears in $\beta_k^{(1,o)}$ is equal to $k+2$. Thus, $d^{(k,k-1)}$ has the structure 
\begin{align}
d^{(k,k-1)} = \left(
\begin{array}{c c}
\tilde{d}^{(k,k-1)} & 0\\
0 & I^{k+2}
\end{array}
\right)\label{struct d} 
\end{align} 
where $\tilde{d}^{(k,k-1)}\in\mbb R^{(n_o(k)-(k+2))\times(n_o(k)-(k+2))}$ and $I^{k+2}$ is an identity matrix of size $(k+2)\times (k+2)$.
From \eqref{simple Dkl}, \eqref{def d} and \eqref{struct d} we can conclude that
\begin{align}
D^{(k,k-1)} = \left(
\begin{array}{c c}
\tilde{d}^{(k,k-1)} & 0\\
0 & I^{k+2}
\end{array}
\right),\quad D^{(k,k+1)} = \left(\begin{array}{c c}
d^{(k,k+1)} ,\quad 0
\end{array}\right). \label{struct D}
\end{align}

\noindent The matrix $\MatFluxBi{M}{M-1}$, which can be constructed by ignoring the contribution from $D^{(M-1,M)}$ into $\MatFluxBi{M}{M}$, is upper triangular with blocks of $D^{(k,k-1)}$ along its diagonal. Since $D^{(k,k-1)}$ contains square roots of odd numbers along its diagonal, which are all non-zero, the invertibility of $\MatFluxBi{M}{M-1}$ follows.

\section{Norms of Matrices and Operators} \label{norm mat and op}
We will need the result
\begin{lemma} \label{norm of sys}
Let $A\in \mbb R^{n\times n}$, $n\geq 1$, be given by
$
A_{ij} = \sqrt{2i-1}\delta_{ij} + \sqrt{2i}\delta_{(i+1)j} $.
Then the solution $x\in\mbb R^n$ to the linear system 
\begin{align}
A_{ij}x_j = \delta_{in} \label{sys}
\end{align}
is such that $\|x\|_{l^2}=1$.
\end{lemma}
\begin{proof}
For $n=1$, the result is trivial and so we consider the $n > 1$ case.  
From the first $n-1$ equations of the linear system \eqref{sys} it follows 
$x_i\sqrt{2i-1} + x_{i+1}\sqrt{2i} = 0$, $ i \in \{1,2,\dots n-1\}$,
with which we can express any $x_p$ ($p\geq 2$) in terms of $x_1$ as
\begin{align}
x_p = (-1)^{p-1}\prod_{k = 1}^{p-1}\sqrt{\frac{2k-1}{2k}}x_1 = (-1)^{p-1}\sqrt{\frac{(2p-3)!!}{(2p-2)!!}}x_1,\quad p\in \{2,\dots n\}. \label{general xp}
\end{align}
Thus
\begin{align}
\|x\|_{l^2}^2 = x_1^2\left(1 + \sum_{p=2}^n\frac{(2p-3)!!}{(2p-2)!!}\right) = x_1^2\sum_{p=0}^{n-1}\frac{1}{2^pp!}. \label{norm x}
\end{align}
From the last equation in 
\eqref{sys} and using \eqref{general xp} we have
$
x_n = 1/{\sqrt{2n-1}}$ which implies\\ $ x_1 = (-1)^{n-1}\sqrt{(2n-2)!!/{(2n-1)!!}}.$ 
Using the expression for $x_1$ in \eqref{norm x}, we find 
\begin{gather*}
\|x\|_{l^2}^2 = \frac{(2n-2)!!}{(2n-1)!!}\sum_{p=0}^{n-1}\frac{1}{2^pp!}.
\end{gather*}
Finally, induction provides
$
\sum_{p=0}^{n-1}1/(2^pp!) = (2n-1)!!/(2n-2)!!$ which implies $\|x\|_{l^2}^2 = 1.
$
\end{proof}
\begin{enumerate}[label = (\roman*)]

\item \textit{Norm of $\lim_{q\to\infty}\MatHalfBi{M}{q}$:}
Let $L = \lim_{q\to\infty}\MatHalfBi{M}{q}$ which is well-defined on $l^2$ due to \autoref{lemma:int split}.
Define $y\in\mbb R^{\Xi_o^M}$ as 
$
y = Lx = 2\lan \CollectBasisO{M}f_0 ,r \ran_{\spaceVPlus}$ where $r = \sum_{m = 0}^{\infty
}x_m \cdot\basisE{m}f_0$, $
x = \left(x_0^',x_1^',\dots,x_k^',\dots  \right)'$ and $x_k\in\mbb R^{\numEven(k)}.
$
Functions $\sqrt{2}\basisE{i}f_0$ are orthonormal under $\lan.,.\ran_{\spaceVPlus}$. This implies $\|r\|^2_{\spaceVPlus} = \frac{1}{2}\|x\|^2_{l^2}$.
Orthogonal projection of $r$ onto $\{\sqrt{2}\basisO{m}f_0\}_{m\leq M}$ can be given as
$
\mcal Pr = \sum_{m = 1}^M y_m\cdot\basisO{m}f_0$ where $
y = \left(y_1^',y_2^',\dots,y_M^' \right)'$ and $y_k\in\mbb R^{\numOdd(k)}.
$
Therefore, it holds 
$
\|\mcal Pr\|_{\spaceVPlus} \leq \|r\|_{\spaceVPlus}.
$
Since 
$
\|\mcal Pr\|_{\spaceVPlus}^2 = \|y\|^2_{l^2}/2$ and $\|r\|^2_{\spaceVPlus} = \|x\|^2_{l^2}/2
$, 
we obtain 
$
\|y\|_{l^2}^2 \leq \|x\|^2_{l^2}
$ which provides $\|L\|\leq 1$.
\item \textit{Norm of $\MatFluxBi{M}{M}$ :}
Let $A = \MatFluxBi{M}{M}\left(\MatFluxBi{M}{M}\right)^'$. Since every row of $\MatFluxBi{M}{M}$ contains two entries, one on the main diagonal and one on the off diagonal (see appendix-\autoref{properties Aoe}), every row of $A$ will contain a maximum of three entries. Since the maximum magnitude of entries in $\MatFluxBi{M}{M}$ is $\mcal O(\sqrt{M})$, the maximum magnitude of the entries, in $A$, will be $\mcal O(M)$. The Gerschgorin's circle theorem then implies that the maximum eigenvalue of $A$ will be $\mcal O(M)$ which implies $\|\MatFluxBi{M}{M}\|_2 \leq C \sqrt{M}$.
\item \textit{Norm of $\|\left(\MatFluxBi{M}{M-1}\right)^{-1}\MatFluxSm{M}{M}\|_2$ :}
In the coming discussion we will assume $M$ to be even; for $M$ being odd, the proof follows along similar lines and will not be discussed for brevity. 
From the definition of $\MatFluxSm{M}{M}$ it is clear that it only has a contribution from $D^{(M-1,M)}\in\mbb R^{\numOdd(M-1)\times\numEven(M)}$, with $D^{(M-1,M)}$ as defined in \eqref{struct D}. 
Let $X \in \mbb R^{\numOddTot{M}\times \numOdd(M-1)}$ represent those columns of $\left(\MatFluxBi{M}{M-1}\right)^{-1}$ which get multiplied with $D^{(M-1,M)}$ appearing in $\MatFluxSm{M}{M}$. As a result 
$
\|\left(\MatFluxBi{M}{M-1}\right)^{-1}\MatFluxSm{M}{M}\|_2 = \|XD^{(M-1,M)}\|_2  \leq \|X\|_2\|D^{(M-1,M)}\|_2.
$
 From \eqref{def d} it follows that $\|D^{(M-1,M)}\|_2\leq C\sqrt{M}$. We show that $X$ is unitary which proves our claim.

\noindent
Let $x^{(\omega)}$ denote the $\omega$-th column of $X$ with $\omega\in \{1,\dots, \numOdd(M-1)\}$. We decompose $x^{(\omega)}$ as
$
x^{(\omega)} = \left(\left(x^{(\omega)}_{n_e(0)}\right)',\left(x^{(\omega)}_{n_e(1)}\right)',\dots,\left(x^{(\omega)}_{n_e(M-1)}\right)'\right)$ where $ x^{(\omega)}_{\numEven(q)}\in \mbb R^{\numEven(q)}.
$
Different values of $x^{(\omega)}$, for different values of $\omega$, can be found by solving the system of equations (which results from $\MatFluxBi{M}{M-1}\left(\MatFluxBi{M}{M-1}\right)^{-1} = I$)
\begin{gather}
D^{(k,k-1)}x^{(\omega)}_{n_e(k-1)} + D^{(k,k+1)}x^{(\omega)}_{n_e(k+1)} = 0\hspB
D^{(M,M-1)}x^{(\omega)}_{n_e(M-1)} = 0,  \label{solveXj1}\\
D^{(M-1,M-2)}x^{(\omega)}_{n_e(M-2)} = I^{\numOdd(M-1)}_{\omega},
 \label{solveXj2}
\end{gather}
where $I^{\numOdd(M-1)}_{\omega}$ is a diagonal matrix of size $\numOdd(M-1)\times \numOdd(M-1)$ such that $\left(I^{\numOdd(M-1)}_{\omega}\right)_{ii}=\delta_{i\omega}$ and $D^{(k,k-1)}$ (and $D^{(k,k+1)}$) are as defined in \eqref{struct D}. From \eqref{solveXj1} we conclude 
$
x^{(\omega)}_{n_e(M-1)} = 0$ which implies $x^{(\omega)}_{n_e(M-(2q-1))} = 0$, $\forall q\in \{1,\dots \frac{M}{2}\}. 
$
We express the set of remaining equations as
\begin{equation}
\begin{aligned}
D^{(k,k-1)}x^{(\omega)}_{n_e(k-1)} + D^{(k,k+1)}x^{(\omega)}_{n_e(k+1)}& = 0,\hsp \forall\hsp k\in \{1,3,\dots,M-3\} \label{reduced solveXj}\\ 
D^{(M-1,M-2)}x^{(\omega)}_{n_e(M-2)} &= I^{\numOdd(M-1)}_{\omega}
\end{aligned}
\end{equation}
Orthogonality of solutions to \eqref{reduced solveXj} is clear from the structure of the linear system itself. Therefore, 
to prove our claim we need to show that 
\begin{gather}
\|x^{(\omega)}\|_{l^2} = 1\hsp \forall\hsp \omega\in\{1,\dots n_o(M-1)\},\label{normality}
\end{gather}
for which we will claim that solving \eqref{reduced solveXj} for a given $\omega$ is equivalent to solving a system of the type \eqref{sys}; the result will then follow from \autoref{norm of sys}. From the entries of $d^{(k,k-1)}$ and $d^{(k,k+1)}$ defined in \eqref{def d}, it follows that the system in \eqref{reduced solveXj} is equivalent to
\begin{align}
\left(\begin{array}{c c c c c c}
1 & \sqrt{2} & 0 & 0 & \dots & \dots\\
0 & \sqrt{3} & \sqrt{4} & 0 & \dots & \dots\\
0 & 0          &  \ddots   & \ddots & 0 & \dots \\
0 & 0 		   & 0 & \dots & \sqrt{(\beta_{M-1}^{(1,o)})_j-2} &  \sqrt{(\beta_{M-1}^{(1,o)})_j-1} \\
0 & 0 		   & 0 & \dots & \dots &  \sqrt{(\beta_{M-1}^{(1,o)})_j} 
\end{array}\right)\left(\begin{array}{c}
\left( x^{(\omega)}_{n_e(M-2q)}\right)_j \\ 
\left( x^{(\omega)}_{n_e(M-2(q-1))}\right)_j \\
\vdots\\
\left( x^{(\omega)}_{n_e(M-2)}\right)_j
\end{array}\right) = \left(\begin{array}{c}
0\\
0\\
0\\
0 \\
\vdots\\
\delta_{j,\omega}
\end{array}\right)  \label{final solveXj}
\end{align}
where $\beta_{k}^{(1,o)}$ is as defined in \autoref{def betako}, $q = \left(\left(\beta_{M-1}^{(1,o)}\right)_j+1\right)/2$ and for every $\omega$,\\  $j\in \{1,\dots ,\numOdd(M-1)\}$. For $j=\omega$, the system in \eqref{final solveXj} is the same as \eqref{sys} and hence \eqref{normality} follows. 
\end{enumerate}

\end{appendices}

\bibliographystyle{abbrv}
\bibliography{paper}

\begin{thebibliography}{10}

\bibitem{abstractKinetic}
R.~Beals and V.~Protopopescu.
\newblock Abstract time-dependent transport equations.
\newblock {\em Journal of Mathematical Analysis and Applications}, 121(2):370
  -- 405, 1987.

\bibitem{PNIntro}
T.~A. Brunner and J.~P. Holloway.
\newblock Two-dimensional time dependent {R}iemann solvers for neutron
  transport.
\newblock {\em Journal of Computational Physics}, 210(1):386 -- 399, 2005.

\bibitem{CaiNRXX}
Z.~Cai and R.~Li.
\newblock Numerical regularized moment method of arbitrary order for
  {B}oltzmann-{B}{G}{K} equation.
\newblock {\em SIAM Journal on Scientific Computing}, 32(5):2875--2907, 2010.

\bibitem{Carlos}
C.~{C}ercignani.
\newblock {\em The Boltzmann Equation and Its Applications}.
\newblock Springer, 67 edition, 1988.

\bibitem{Christian}
R.~{C}hristian.
\newblock Numerical methods for the semiconductor {B}oltzmann equation based on
  spherical harmonics expansions and entropy discretizations.
\newblock {\em Transport Theory and Statistical Physics}, 31(4-6):431--452,
  2002.

\bibitem{Douglas78}
J.~Douglas, T.~Dupont, and M.~F. Wheeler.
\newblock A quasi-projection analysis of {G}alerkin methods for parabolic and
  hyperbolic equations.
\newblock {\em Mathematics of Computation}, 32(142):345--362, 1978.

\bibitem{CutOffQ}
H.~B. Drange.
\newblock The linearized {B}oltzmann collision operator for cut-off potentials.
\newblock {\em SIAM Journal on Applied Mathematics}, 29(4):665--676, 1975.

\bibitem{Douglas73}
T.~Dupont.
\newblock L2-estimates for {G}alerkin methods for second order hyperbolic
  equations.
\newblock {\em SIAM Journal on Numerical Analysis}, 10(5):880--889, 1973.

\bibitem{Egger2012}
H.~Egger and M.~Schlottbom.
\newblock A mixed variational framework for the radiative transfer equation.
\newblock {\em Mathematical Models and Methods in Applied Sciences},
  22(03):1150014, 2012.

\bibitem{Egger2016}
H.~Egger and M.~Schlottbom.
\newblock A class of galerkin schemes for time-dependent radiative transfer.
\newblock {\em SIAM Journal on Numerical Analysis}, 54(6):3577--3599, 2016.

\bibitem{ExistenceLinearisedBE}
L.~Falk.
\newblock Existence of solutions to the stationary linear {B}oltzmann equation.
\newblock {\em Transport Theory and Statistical Physics}, 32(1):37--62, 2003.

\bibitem{Martin}
M.~Frank, C.~Hauck, and K.~Kupper.
\newblock Convergence of filtered spherical harmonic equations for radiation
  transport.
\newblock {\em Commun. Math. Sci}, 14(5):1443--1465, 2016.

\bibitem{GambaPetrov}
I.~M. Gamba and S.~Rjasanow.
\newblock Galerkin-{P}etrov approach for the {B}oltzmann equation.
\newblock {\em Journal of Computational Physics}, 366:341 -- 365, 2018.

\bibitem{Grad1949}
H.~{G}rad.
\newblock {On the kinetic theory of rarefied gases}.
\newblock {\em Communications on Pure and Applied Mathematics}, 2(4):331--407,
  1949.

\bibitem{GradRGD}
H.~Grad.
\newblock Asymptotic theory of the {B}oltzmann equation. {II}.
\newblock {\em Pros. 3rd Internat. Sympos., Palais de l'UNESCO, Paris, 1962},
  1:26--59, 1962.

\bibitem{GradAsymp}
H.~Grad.
\newblock Asymptotic theory of the {B}oltzmann equation.
\newblock {\em The Physics of Fluids}, 6(2):147--181, 1963.

\bibitem{BCLayerBroadwell}
J.-G. Liu and Z.~Xin.
\newblock Boundary-layer behavior in the fluid-dynamic limit for a nonlinear
  model {B}oltzmann equation.
\newblock {\em Archive for Rational Mechanics and Analysis}, 135(1):61--105,
  Oct 1996.

\bibitem{LucDVM}
L.~Mieussens.
\newblock Discrete-velocity models and numerical schemes for the boltzmann-bgk
  equation in plane and axisymmetric geometries.
\newblock {\em Journal of Computational Physics}, 162(2):429 -- 466, 2000.

\bibitem{Rana2016}
A.~S. {R}ana and H.~{S}truchtrup.
\newblock {Thermodynamically admissible boundary conditions for the regularized
  13 moment equations}.
\newblock {\em Physics of Fluids}, 28(2):027105, 2016.

\bibitem{MomentsToDVM}
C.~Ringhofer, C.~Schmeiser, and A.~Zwirchmayr.
\newblock Moment methods for the semiconductor {B}oltzmann equation on bounded
  position domains.
\newblock {\em SIAM Journal on Numerical Analysis}, 39(3):1078--1095, 2001.

\bibitem{Sarna2018}
N.~Sarna and M.~Torrilhon.
\newblock Entropy stable {H}ermite approximation of the linearised {B}oltzmann
  equation for inflow and outflow boundaries.
\newblock {\em Journal of Computational Physics}, 369:16 -- 44, 2018.

\bibitem{Sarna2017}
N.~Sarna and M.~Torrilhon.
\newblock On stable wall boundary conditions for the {H}ermite discretization
  of the linearised {B}oltzmann equation.
\newblock {\em Journal of Statistical Physics}, 170(1):101--126, Jan 2018.

\bibitem{ConvergenceMoments}
C.~Schmeiser and A.~Zwirchmayr.
\newblock Convergence of moment methods for linear kinetic equations.
\newblock {\em SIAM Journal on Numerical Analysis}, 36(1):74--88, 1998.

\bibitem{Struchtrupbook}
H.~{S}truchtrup.
\newblock {\em {Macroscopic Transport Equations for Rarefied Gas Flows}}.
\newblock Springer Ltd, 2010.

\bibitem{HermiteSobolev}
S.~Thangavelu.
\newblock On regularity of twisted spherical means and special {H}ermite
  expansions.
\newblock {\em Proceedings of the Indian Academy of Sciences - Mathematical
  Sciences}, 103(3):303, Dec 1993.

\bibitem{Torrilhon2015}
M.~{T}orrilhon.
\newblock Convergence study of moment approximations for boundary value
  problems of the {B}oltzmann-{BGK} equation.
\newblock {\em Communications in Computational Physics}, 18(03):529--557, 2015.

\bibitem{Torrilhon2017}
M.~Torrilhon and N.~Sarna.
\newblock Hierarchical {B}oltzmann simulations and model error estimation.
\newblock {\em Journal of Computational Physics}, 342:66 -- 84, 2017.

\bibitem{Ukai}
S.~Ukai.
\newblock Solutions of the {B}oltzmann equation.
\newblock In {\em Patterns and Waves}, volume~18 of {\em Studies in Mathematics
  and Its Applications}, pages 37 -- 96. Elsevier, 1986.

\end{thebibliography}
\end{document}